\documentclass[onefignum,onetabnum]{siamart190516}

\usepackage{tabulary}

\RequirePackage{fix-cm}
\usepackage{graphicx}
\usepackage{psfrag}
\usepackage{subfigure}
\usepackage{url}
\usepackage{color}
\usepackage{cite}
\usepackage{epsfig}
\usepackage{multirow}
\usepackage{amsmath,amssymb}
\usepackage{algorithmic}

\usepackage{amsfonts,mathrsfs}
\usepackage{verbatim}
\usepackage{acronym}
\usepackage{array}
\usepackage[utf8]{inputenc}
\usepackage[english]{babel}
\usepackage[mathscr]{euscript}
\usepackage{longtable}
\usepackage{footnote}
\usepackage{bbm}

\newcommand{\red}{\textcolor{black}}
\newcommand{\blue}{\textcolor{black}}

\newcommand{\rset}{\mathbb{R}}
\newcommand{\Hf}{H_{f,U_{i_k}}}
\newcommand{\Hfi}{H_{f,U_{i}}}

\newcommand{\HF}{H_{F_{k}}}
\newcommand{\HFi}{H_{F_{i}}}
\newcommand{\HfU}{H_{f,U_{i}}}
\newcommand{\HfW}{H_{f,W}}
\newcommand{\HfV}{H_{f,V}}

\newcommand{\Hfmax}{H_{f,\max}}
\newcommand{\HFmax}{H_{F,\max}}
\newcommand{\HFmaxc}{\bar{H}_{F,\max}}

\newcommand{\Hpsimaxb}{\bar{H}_{\psi,\max}}
\newcommand{\Hp}{H_{\psi}}
\newcommand{\inas}{\overset{\text{a.s.}}{\to}}

\newsiamremark{remark}{Remark}
\newtheorem{assumption}[theorem]{Assumption}

\newsiamremark{example}{Example}

\title{Random coordinate descent  methods for nonseparable composite  optimization 
}

\author{Flavia Chorobura\thanks{Automatic Control and Systems
	Engineering Department, University Politehnica Bucharest, 060042
	Bucharest, Romania, \email{flavia.chorobura@stud.acs.upb.ro.}}
	\and Ion Necoara\thanks{Automatic Control and Systems
	Engineering Department, University Politehnica Bucharest, 060042
	Bucharest, Romania  and  Gheorghe Mihoc-Caius Iacob  Institute of Mathematical Statistics and Applied Mathematics of the Romanian Academy, 050711 Bucharest, Romania,  \email{ion.necoara@upb.ro.}}}

\begin{document}

\maketitle

\begin{abstract}
In this paper we consider large-scale composite  optimization problems having  the  objective function formed as a sum of two  terms (possibly  nonconvex), one has \red{(block) coordinate-wise} Lipschitz continuous gradient and the other is  differentiable but nonseparable.  Under these general settings we derive and analyze two new coordinate descent methods.  The first algorithm, referred to as coordinate proximal gradient method, considers the composite form of the objective function, while the other algorithm disregards the composite form of the objective  and uses the partial gradient of the full objective, yielding a coordinate gradient descent scheme with novel adaptive stepsize rules. We prove that these new stepsize rules make the coordinate gradient  scheme a descent method, provided that additional assumptions hold for the second term in the objective function.  We  present a complete worst-case complexity analysis for these two new methods in both, convex and nonconvex settings,  provided that the (block) coordinates are chosen random or cyclic. Preliminary numerical results also  confirm the  efficiency of our  two algorithms on practical problems.   
\end{abstract}




\section{Introduction}
\noindent In this paper we consider solving  large-scale  composite optimization problems of the form: 
\vspace*{-0.05cm}
\begin{equation}
\label{eq:prob}
F^* = \min_{x \in \mathbb{R}^n} F(x)  \quad \left(:= f(x) + \psi(x)\right),
\vspace{-0.05cm}
\end{equation}
where the function $f: \mathbb{R}^n  \to \mathbb{R}$  has block coordinate-wise Lipschitz gradient   and $\psi : \mathbb{R}^n  \to {\mathbb{R}}$ is twice  continously  differentiable function (both terms are possibly \textit{nonseparable} and  \textit{nonconvex}).  Optimization problems having this composite structure permit to handle general coupling functions $\psi$ (e.g., $\psi(x) = \|Ax\|^p$, with A linear operator and $p \geq 2$) and arise in many applications such as  distributed control, signal  processing,  machine learning,  network flow problems and other areas \cite{Ber:99,Mit:97,NecCli:13}. Despite  the bad properties of the sum (nonsmoothness), such problems, both in convex and nonconvex cases, can be solved by full gradient or Newton methods with the efficiency typical for the good (smooth) part  of the objective  \cite{NesPol:06}.  However, for large-scale problems,  the usual methods based on full gradient and Hessian  computations are prohibitive. In this case, it appears that a reasonable approach to solve such problems   is to use  (block) coordinate descent methods. 

\medskip 

\noindent \textit{State of the art}.  Coordinate (proximal) gradient  descent   methods   \cite{BecTet:13, BolSab:14,FerRic:15,LuXia:14,Nes:10,NecCli:16,NecTak:20,NecCho:21,NesSti:17, RicTak:11,TseYun:09}, see also the recent survey \cite{Wri:15}, gained attention in optimization in the last years due to their fast convergence and small computational cost per iteration.  The main differences in all variants of coordinate descent algorithms consist in the way we define the local approximation
function over which we optimize  and the criterion of choosing at each iteration the coordinate over which we minimize this local approximation.  For updating one (block) variable, while keeping the other variables fixed,  two basic choices for the local approximation   are usually considered: (i) exact approximation function, leading to \textit{coordinate minimization methods} \cite{Bec:14, GriSci:00} and (ii)  quadratic approximation function, leading to \textit{coordinate gradient descent methods} \cite{Nes:10,TseYun:09, Wri:15}.  Furthermore, three criteria for choosing the coordinate search  used often in these algorithms are the greedy, the cyclic and the random coordinate search, respectively. For cyclic coordinate search convergence rates have been given recently in  \cite{Bec:14, BecTet:13}. Convergence rates for coordinate descent methods based on the Gauss-Southwell rule  were derived  in \cite{TseYun:09}. Another interesting approach is based on random coordinate descent, where the coordinate search is random. Complexity results on random coordinate descent methods were obtained  in \cite{Nes:10}  for smooth convex functions. The extension to composite objective functions  were given in \cite{FerRic:15, NesSti:17, NecCli:16,  RicTak:11}. \red{However, these  papers studied optimization models where the second term, usually assumed nonsmooth, is separable,  i.e., $\psi(x) =\sum_{i=1}^n \psi_i(x_i)$, with $x_i$ is the $i$th  component of $x$.  In the sequel, we discuss papers that consider the case $\psi$  nonseparable and explain the main \blue{differences} with our present work. }

\medskip 

\noindent \blue{\textit{Previous work}}.  From our knowledge there exist  few studies  on  coordinate descent methods when the second term in the objective function is nonseparable. For example,  \cite{Nec:13,NecTak:20,TseYun:09} considers  the composite optimization problem \eqref{eq:prob}  with $\psi$ convex and  separable (possibly nonsmooth) and the additional  nonseparable constraints $Ax=b$.  Hence,   nonseparability comes from the linear constraints.  In these settings, \cite{Nec:13,NecTak:20,TseYun:09}  proposed coordinate proximal gradient descent methods that require solving at each iteration  a subproblem over a  subspace generated by the matrix $U \in \mathbb{R}^{n \times p}$ using  a part of the gradient of $f$ at the current feasible  point $x$, $\nabla f(x)$, i.e:   
	\vspace*{-0.2cm}
\begin{equation*}
\min_{d \in \mathbb{R}^p} f(x) + \langle  U^T \nabla f(x),d \rangle + \frac{1}{2} d^T H_U d + \psi(x) \quad \text{s.t.} \quad AUd=0,
	\vspace*{-0.2cm}
\end{equation*}
where $H_U$ is an appropriate positive definite  matrix  and then update $x^+ = x + Ud$.  The matrix $U$ is chosen according to some greedy rule or random. For these algorithms  sublinear rates are derived in the (non)convex case and linear convergence is obtained for  strongly convex objective. Further, for  problem    \eqref{eq:prob}, with $\psi$ possibly nonseparable and nonsmooth,  \cite{GriIut:21,HanKon:18, LatThe:21} consider proximal coordinate descent  methods  of the form:
	\vspace*{-0.2cm} 
\begin{equation}
\bar{x}^{+} \in  \text{prox}_{\alpha \psi } (\mathcal{C}(x - \alpha \nabla f(x))), 
	\vspace*{-0.2cm}
\end{equation}  
\noindent where $\mathcal{C}( \cdot)$   is a correction map  corresponding to the chosen random subspace at the current iteration in \cite{GriIut:21,HanKon:18} and is the identity map in \cite{LatThe:21}.  Moreover,  \cite{GriIut:21,HanKon:18} assume $\psi$ convex and update $x^{+} = \bar{x}^{+}$, while in \cite{LatThe:21} $\psi$ is possibly nonconvex and updates   $x_{i}^{+} = \bar{x}_{i}$  for all $i \in \mathcal{I}  \subseteq [n]$ and keeps the rest of the components  unchanged.  Note that  \cite{GriIut:21,HanKon:18}   use only  a sketch of  the gradient   $\nabla f(x)$ on the selected subspace, while in  \cite{LatThe:21}  $f$ is assumed separable.  
\red{Since in these papers \cite{GriIut:21,HanKon:18,LatThe:21},  one needs to compute at each iteration a block of components of the full prox of the nonseparable function $\psi$,  $\text{prox}_{\alpha \psi }$,  this can be can done efficiently  when this prox  can be evaluated easily based on the previously computed prox  and providec that only a block of coordinates are modified in the prior iteration.}
For the  algorithms in \cite{GriIut:21,HanKon:18}  linear convergence is derived, provided that the objective function  is  strongly convex. Linear convergence results were also obtained in \cite{LatThe:21}  when the objective function satisfies the Kurdyka-Lojasiewicz  condition.  Recently, \cite{AbeBec:21} considers  problem \eqref{eq:prob},  where the function  $f$ is assumed quadratic and convex,  while $\psi$ convex   function (possibly nonseparable and nonsmooth).  Under these settings,  \cite{AbeBec:21}  combines  the forward-backward envelope (to  smooth the original problem)  with an accelerated coordinate gradient descent method and derives sublinear rates for the proposed scheme. \red{ This method also makes sense when the full prox can be computed efficiently under coordinate descent updates. The main difference between  our work and  \cite{GriIut:21,HanKon:18,LatThe:21,AbeBec:21} is that in our first algorithm we consider a prox along coordinates,  while in the other papers  one needs to compute a block of components of the full prox. Moreover, in the second algorithm our search direction is based on the partial gradient of the full objective function. }
	
\medskip 

\noindent The paper most related to  the first algorithm  is  \cite{NecCho:21}.  More precisely,  in  \cite{NecCho:21} at each iteration one needs to sketch  the gradient   $\nabla f$  and compute the  prox of $\psi$  along some subspace  generated by the random matrix $U \in \mathbb{R}^{n\times p}$, that is:
	\vspace*{-0.1cm}
\begin{equation}  
\label{eq:fullprox1}	
 x^{+} = \text{prox}_{ H_{f,U}^{-1}  \phi}  \left( - H_{f,U}^{-1} \;  U^{T} \nabla f(x) \right),  	
 	\vspace*{-0.2cm}
\end{equation} 
where $\phi(d)  =  \psi(x + Ud)$.  Assuming that $\psi$ is twice differentiable,  (sub)linear convergence rates are derived in \cite{NecCho:21} for both convex and nonconvex settings. However, depending on the properties of the random matrix $U$, in each iteration we need to update a block of components of  $x$, whose dimension $p$, in some cases, \blue{may depend on $n$.}
In this paper we also design for the  composite problem \eqref{eq:prob}  a random coordinate proximal gradient method of the form \eqref{eq:fullprox1}	 that  uses a block of components of the gradient   $\nabla f$ and requires  the computation of the  prox of $\psi$  along these coordinates. However, in this algorithm we do not have restrictions on the subspace dimension, in the extreme case we can update only one component of $x$.  In this paper, we also propose a  second algorithm,  which contrary to the usual approach from literature, disregards the composite form of the objective function and makes an update based on  the partial gradient of the full objective function:
	\vspace*{-0.1cm}
\begin{equation}
	\label{eq:fullprox2}	
x^{+} = x - H_{F,U}^{-1}   U^T \nabla F(x).
	\vspace*{-0.1cm}
\end{equation}
\noindent We propose several new adaptive stepsize rules, $H_{F,U}$, based on some additional assumptions on the second term $\psi$. 

\medskip 

\noindent \textit{Contributions}. This paper deals with large-scale composite  optimization problems of the form  \eqref{eq:prob}. We present  two coordinate descent methods,  \eqref{eq:fullprox1}	and 	\eqref{eq:fullprox2},  and derive convergence rates when the (block) coordinates are chosen  random or   cyclic.  More precisely, our  contributions are:\\
(i) We introduce a coordinate proximal gradient method,  \eqref{eq:fullprox1}, which takes into account the nonseparable composite form of the objective function. In each iteration,  one needs to compute a block of components of the gradient $\nabla f$, followed by the prox of $\psi$ along this block of coordinates. Note that typically, the prox restricted to some subspace leads to much less computations than the full prox. \\
(ii) We also present a coordinate gradient method, 	\eqref{eq:fullprox2}, which requires at each iteration the computation of a block of components of the gradient of the full objective function, i.e., $U^{T}\nabla F$. We propose new stepsizes strategies for this method, which guarantees descent and convergence under certain assumptions on $\psi$.  In particular, three of these stepsize rules are \emph{adaptive} and require computation of  a positive root of a polynomial, while the last one can be chosen constant.   \\
(iii)  We derive sufficient conditions for the iterates of our algorithms to be bounded. We also prove that our algorithms are descent methods and derive  sublinear convergence rates,  provided that the (block) coordinates are chosen  random or cyclic, in the convex and nonconvex settings.  Improved rates are given under Kurdyka-Lojasiewicz (KL) property, i.e., sublinear or linear depending on the KL parameter. The convergence rates obtained in this paper are summarized in Table \ref{table:rates}.  Since  uniform convex  functions satisfy  KL property, our rates also cover this~case. 

\medskip 
	
\noindent 	Note that in this paper we perform a full convergence  analysis for a  random coordinate descent algorithm  for solving general  (non)convex composite problems and most of our variants of coordinate descent schemes  were never explicitly considered in the literature before. \red{Although our algorithms belong to the class of coordinate gradient descent methods, our  convergence results are also of  interest when $f \equiv 0$ and $\psi$ nonseparable (in this case our first algorithm can be viewed as a proximal regularization of a multi-block Gauss-Seidel method)}. In particular, this is the first work where  convergence bounds are presented for an exact coordinate minimization (Gauss-Seidel) method, i.e., when  $f \equiv 0$,   and for a   coordinate gradient descent method, i.e., when the full function $F$ doesn't have coordinate-wise Lipschitz gradient, in both convex and nonconvex~settings. Recall  that 
if $\psi$ is nonseparable, coordinate descent methods may not converge (see e.g., the counterexamples in \cite{FriHasHofTib} for  nonseparable nondifferentiable convex problems   and in \cite{Ber:99,Pow:73} for nonseparable nonconvex problems, even in the differentiable case).  These results motivate us to consider $\psi$ twice differentiable.
  \vspace*{-0.2cm}
 
\begin{footnotesize}
	\begin{table}[h!]
		\begin{center}
		\begin{tabular}{|p{1.7cm}  |p{4.4 cm}| p{0.8cm}| p{3.1cm}| p{0.8cm}|}
			\hline
			\multicolumn{5}{|c|}{Random} \\
			\hline
		   \multirow{1}{*}{Nonconvex} & \begin{footnotesize}
		   	 \red{$\min\limits_{i=0:k-1} \! \mathbb{E} \left[ \|\nabla F(x_{i})\|\right]  \!\leq\! \mathcal{O} \! \left(\!Nk^{-\frac{1}{2}} \!\right)$} \end{footnotesize}  &  \small{Rem.} \ref{rem:gradR}  &  \multirow{2}{3.1cm}{\begin{center} \vspace*{-0.3cm} \begin{small} $\forall \delta>0$,  with prob. $1-\delta$: $\mathbb{E} \left[ F(x_k) \right] \to F_*$  sublinearly or linearly  \end{small}\end{center}} & \multirow{4}{0.8cm}  {\begin{center}
		   		 	\vspace{-0.3cm}  
		   		 	\small{Thm.} \ref{theo:KL1}
		   		 \end{center}
		   }  \\
		\cline{1-3}
			Convex  & \red{ \(\displaystyle \mathbb{E} \left[ F(x_k)\right]  - F^* \!\leq\! \mathcal{O}\left(Nk^{-1}\right) \) }  & \small{Thm.} \ref{the:conv}   & & \\
			\hline
			\multicolumn{5}{|c|}{Cyclic} \\
			\hline
			\multirow{1}{*}{Nonconvex} & \begin{small}  \red{$\min\limits_{i=0:k-1} \!   \|\nabla F(x_{i})\|  \!\leq\! \mathcal{O} \! \left(\! N^2k^{-\frac{1}{2}} \!\right)$} \end{small}  &  \small{Rem.} \ref{rem:gradC}  &  \multirow{4}{3.1cm}{\begin{center} \vspace*{-0.2cm} \begin{small}  $F(x_{kN}) \to F_*$  sublinearly, linearly \red{or superlinearly} \end{small}\end{center}} & \multirow{4}{0.8cm}  {\begin{center}
					\vspace{-0.3cm} \small{Thm.}  \ref{the:KLC}
				\end{center}
			}\\
			\cline{1-3}
			Convex  &  \red{ \(\displaystyle  F(x_{kN}) - F^* \!\leq\! \mathcal{O}\left(Nk^{-1}\right) \) }  & \small{Thm.}  \ref{the:conC}   & & \\
			\hline
		\end{tabular}
	\end{center}
	\caption{Convergence rates derived in this paper for the algorithms \eqref{eq:fullprox1}	and 	\eqref{eq:fullprox2}.}
\label{table:rates}
\end{table}
\end{footnotesize}

\vspace{-0.5cm}

\noindent \textit{Content}. The paper is organized as follows. In Section 2 we present some  preliminary results. We derive in Section 3 the  coordinate proximal gradient algorithm, while in Section 4  the  coordinate gradient  algorithm. In Section 5 we present sufficient conditions for the iterates of our algorithms  to be bounded. The convergence rates  in the random and cyclic cases are derived in Section 6 for the nonconvex case and in Section 7 for the convex case. Finally, in Section 8 we provide detailed numerical~simulations.


\section{Preliminaries}
\noindent In this section we present some definitions,  some preliminary results  and our basic assumptions for the composite problem  \eqref{eq:prob}.
\subsection{Assumptions/setup}  
\noindent We consider the following problem settings.  Let $U \in \mathbb{R}^{n\times n}$ be  a column permutation of the  identity matrix $I_n$ and further let $U = [U_{1},...,U_{N}]$ be a decomposition of $U$ into $N$ submatrices, with $U_{i} \in \mathbb{R}^{n \times n_{i}}$ and  $\sum_{i = 1}^{N}n_{i} = n$.  Hence, any  $x \in \rset^n$ can be written  as $x = \sum_{i = 1}^{N} U_{i} x^{(i)}$, where $x^{(i)} = U_{i}^{T}x \in \mathbb{R}^{n_{i}}$. 
Throughout the paper the following assumptions hold:
\begin{assumption}
	\label{ass1}
	For composite optimization problem \eqref{eq:prob} the following hold: \\
A.1:  Gradient of $f$ is block coordinate-wise Lipschitz continuous with constants $L_i$:
		\begin{align}
		\label{lip1}
		\|U_{i}^{T}(\nabla f(x + U_{i} h) - \nabla f(x)) \| \leq L_{i} \|h\| \quad  \forall h \in \mathbb{R}^{n_{i}},  x \in \mathbb{R}^{n}, i = 1:N.
		\end{align}
A.2:  $\psi$ is twice continuously differentiable (possibly nonseparable and nonconvex). \\
A.3: A solution exists for \eqref{eq:prob}  (hence, the optimal value $F^* >- \infty$). 
\end{assumption}	

\noindent If  Assumption \ref{ass1}[A1] holds, then we have the relation \cite{Nes:10}:
\vspace*{-0.01cm}
\begin{equation}
\label{lip2}
|f(x+U_{i} h) - f(x) - \langle U_{i}^{T} \nabla f(x), h \rangle| \leq \frac{L_{i}}{2} \|h\|^2 \quad \forall h \in  \mathbb{R}^{n_{i}}, \quad i = 1,\cdots,N.
\vspace*{-0.1cm}
\end{equation}

\noindent The basic idea of our algorithms consist of choosing $i \in \{1,...,N\}$ uniformly at random or cyclic and update $x\in\mathbb{R}^{n}$ as follows: 
$
x^+ = x+ U_{i}d. 
$ We consider two choices for the directions $d$. In Coordinate Proximal Gradient (CPG) algorithm, the direction $d$ is computed by a proximal operator of $\psi$ restricted to the subspace $U_{i}$. In Coordinate Gradient Descent (CGD) algorithm, $d$ is given by a multiple of a block of components of the gradient $\nabla F(x_{k})$.

\begin{definition}
\label{definition:3}
For any fixed $x \in \mathbb{R}^{n}$ and $i =1:N$ denote $\phi^{x}_{i}:\mathbb{R}^{n_{i}}\to \mathbb{R}$ as:
\vspace*{-0.1cm}
\begin{equation}
\phi^{x}_{i}(d) = \psi(x + U_{i}d). \label{eq:phi}
\vspace*{-0.1cm}
\end{equation}
 We say that the function $\psi:\mathbb{R}^{n}\to \mathbb{R}$ is convex along coordinates if the partial functions $\phi^{x}_{i}:\mathbb{R}^{n_{i}}\to \mathbb{R}$ are convex for all $x \in \mathbb{R}^{n}$ and $i=1:N$.
\end{definition}

\noindent  One can easily notice that there are nonconvex functions $\psi$ which are convex along coordinates.	 Note that if $\psi$ is twice differentiable, then it is convex along coordinates if $U_{i}^{T}\nabla^2 \psi(x) U_{i}$ is positive semidefinite matrix for any $x$ and $U_{i} \in \mathbb{R}^{n\times n_{i}}$, with $i=1:N$.
  
\noindent Below, we  use the following   mean value inequality (see Appendix for a proof). 
\begin{lemma}
	\label{lemma:MVI}
	  Let $G:  \mathbb{R}^{n} \to \mathbb{R}^{m}$ be a continuously differentiable function and $J:  \mathbb{R}^{n} \to \mathbb{R}^{m\times n}$ be its Jacobian. Consider $U\in\mathbb{R}^{n\times r}$ a fixed matrix  and $x, x+Ud \in \mathbb{R}^{n}$, with $d \in \mathbb{R}^{r}$. Then,  there exists $y  \in  [x,x+Ud]$ such that:
	  \vspace*{-0.2cm}
	\begin{equation*}
		\|G(x+Ud) - G(x)\| \leq \|J(y)U\| \|d\|.
	\end{equation*}
\end{lemma}

\subsection{KL property}
Let us recall the definition of the Kurdyka-Lojasiewicz (KL) property for a function, see e.g.,  \cite{BolDan:07}. Note that the KL property is defined for general functions (possibly nondifferentiable).  Below, we adapt this definition to the differentiable case, since in this paper we consider  only  differentiable objective  functions.

\begin{definition}
	\label{def:kl}
	\noindent A differentiable  function $F$ satisfies  KL property on a compact set $\Omega$ on which $F$ takes a constant value $F_*$ if there exist $\gamma, \epsilon >0$ such that   one has:
	\vspace*{-0.2cm}
	\begin{equation*}
	\kappa' (F(x) - F_*)\| \nabla F(x) \|  \geq 1  \quad   \forall x\!:  \text{dist}(x, \Omega) \leq \gamma, \;  F_* < F(x) < F_* + \epsilon,  
	\vspace*{-0.1cm}
	\end{equation*}
	where $\kappa: [0,\epsilon] \to \mathbb{R}$ is a concave differentiable function satisfying $\kappa(0) = 0$ and  $\kappa'>0$.  
\end{definition}

\noindent The KL property  holds for semi-algebraic functions (e.g., real polynomial functions), vector or matrix (semi)norms (e.g., $\|\cdot\|_p$ with $p \geq 0$ rational number), logarithm functions,  exponential functions and  uniformly convex functions,  see \cite{BolDan:07} for a comprehensive list.


\section{A coordinate proximal gradient algorithm}
\noindent  In this section we assume that the function $\psi$ is simple, i.e.,  $\psi$ restricted to  any subspace generated by $U_{i} \in \mathbb{R}^{n\times n_{i}}$ is proximal easy.    For minimizing the composite  problem \eqref{eq:prob},   where $f$ and $\psi$ are  possibly nonseparable and nonconvex, we propose    a pure  coordinate proximal gradient algorithm that requires some block of components of the gradient $\nabla f(x)$ and computes the  prox of $\psi$ also along these block of coordinates.  Hence, our Coordinate Proximal Gradient (CPG) algorithm is as follows:

\begin{center}
	\noindent\fbox{%
		\parbox{12cm}{%
			\textbf{Algorithm 1 (CPG)}:\\
			Given a starting point $x_{0} \in \rset^n$.  For $k \geq 0$ do: \\
			1. \; Choose $i_{k} \in \{1,...,N\}$ uniformly at random or cyclic and $\eta_{i_{k}}>0$. Set:
			\begin{equation}
			\label{eq:H}
			\Hf = \left\lbrace\begin{array}{ll} \dfrac{L_{i_{k}} + \eta_{i_{k}}}{2} \; 
			\text{ if } \psi \; \text{convex along  coordinates} \\
			L_{i_{k}} + \eta_{i_{k}} \; \text{ otherwise}  
			\end{array}\right.   
			\end{equation}			
			2. \;  Find $d_k$ solving the following subproblem:
			\vspace*{-0.3cm}
			\begin{equation}
			\label{eq:subproblem}
			\;\; d_{k} \in  \arg \min_{d \in \mathbb{R}^{n_{i_{k}}}} f(x_{k}) +  \langle U^{T}_{i_k} \nabla f(x_{k}), d \rangle + \dfrac{\Hf}{2} \|d\|^2  + \psi(x_{k} + U_{i_k}d)	
			\vspace*{-0.1cm}
			\end{equation}
			3. \; 	Update $x_{k+1} = x_{k} + U_{i_k}d_{k}$.
	}}
\end{center}

\medskip
\noindent Note that for $U_{i_k} = I_{n}$, CPG recovers the \textit{full} proximal gradient method, algorithm (46) in \cite{Nes:19}, while for $f \equiv 0$ we get a Gauss-Seidel type algorithm similar to \cite{GriSci:00}.   However,  \cite{Nes:19} derives rates only in the convex settings, while  there are very few results ensuring that the iterates of a Gauss-Seidel method converges to a global minimizer, even for strictly convex functions, e.g., \cite{GriSci:00} presents only assymptotic convergence results. In this paper we derive convergence rates for the general algorithm CPG in both convex and nonconvex settings. An important fact concerning our approach is that the convergence of  CPG  works for any  $\eta_{i_k}$ greater than a fixed positive parameter which can be chosen arbitrarily small. In particular, in  CPG  we can choose a larger stepsize when $\psi$ is convex along coordinates (see Definition \ref{definition:3}), since $\Hf$ must  satisfy in this case  $\Hf> L_{i_{k}}/2$. \red{ When $\psi$ is $\rho$-weakly convex along coordinates,  the subproblem \eqref{eq:subproblem} is convex, provided that $\HfU \geq \rho$}. Our algorithm requires computation of the  proximal operator only of the \textit{partial} function $ \phi_{i_{k}}^{x_{k}}$ (defined in \eqref{eq:phi}) at $U^{T}_{i_k} \nabla f(x_{k})$:
\vspace*{-0.1cm}
\begin{equation}
	\label{prox3.2}  
	d_k\in  \text{prox}_{ \Hf^{-1}  \phi_{i_{k}}^{x_{k}}}  \left( - \Hf^{-1} \;  U^{T}_{i_k} \nabla f(x_{k}) \right). 
	\vspace*{-0.1cm} 	  
\end{equation} 

\noindent  Regardless of the properties of the two functions $f$ and $\psi$, the subproblem  \eqref{eq:subproblem} in CPG is convex provided that $\psi$ is (weakly) convex along coordinates and then the  prox operator 	\eqref{prox3.2}   is well-defined (and unique) in this case, while for general nonconvex $\psi$, the prox 	\eqref{prox3.2}  has to be interpreted as a point-to-set mapping. The proximal mapping is available in closed form for many useful functions, e.g., for  norm power $p$ regularizers. Note that the prox restricted to some subspace  (as required in CPG) is much less expensive computationally than the full prox (as required in the literature  \cite{GriIut:21,HanKon:18, LatThe:21,Nes:19}). \noindent More precisely,  if $\psi$ is differentiable,  then solving the subproblem e.g., in the full proximal gradient method (algorithm (46) in \cite{Nes:19}), is equivalent to finding a full vector $d_{k} \in \mathbb{R}^{n}$ satisfying the system of $n$ nonlinear equations:
	\vspace*{-0.1cm}
	\begin{equation}
		\nabla f(x_{k}) + \nabla \psi(x_{k} + d_k) +H_f d_{k}  = 0. \label{systEq}
		\vspace*{-0.1cm}
	\end{equation}
	
	\noindent  On other hand, when $N=n$ and $U_{i} = e_{i}$, where $e_{i}$ is the $i$th vector of the canonical basis of $\mathbb{R}^{n}$,   at each iteration of  our algorithm CPG,  solving the subproblem \eqref{eq:subproblem} is \red{equivalent} to finding a scalar $d_{k} \in \mathbb{R}$ satisfying the scalar nonlinear equation:
	\vspace*{-0.1cm}
	\begin{equation}
		e^{T}_{i_k} (\nabla f(x_{k}) + \nabla \psi(x_{k} + d_k e_{i_k})) +\Hf d_{k}  = 0.   \label{oneEq}
		\vspace*{-0.1cm}
	\end{equation}
	\noindent  Clearly, there are very efficient methods for finding the root of a scalar equation  \eqref{oneEq},  while it can be  more difficult  to solve the  system of nonlinear equations \eqref{systEq}.

\medskip 

\noindent Next,  we prove that algorithm CPG is a descent method provided that the smooth function $f$ is nonconvex and nonseparable,  and $\psi$ is simple, but possibly nonseparable, nonconvex and twice differentiable.  
Let us denote:
\vspace*{-0.1cm}
\begin{equation}
\eta_{\min} = \min_{i_k=1:N } \eta_{i_{k}} \quad \text{and} \quad \Hfmax  =   \max_{i_k=1:N} \Hf. \label{eq:etaH}
\end{equation}

\begin{lemma}
	\label{lem1}
	If  Assumption \ref{ass1} holds, then   iterates of CPG satisfy the descent:
	\vspace*{-0.1cm}
	\begin{align}
	\label{descent}
	F(x_{k+1}) \leq F(x_{k}) - \dfrac{\eta_{\min}}{2} \|d_{k}\|^2  \quad \forall k \geq 0. 
	\end{align}
\end{lemma}

	\vspace{-0.2cm}
	
\begin{proof}
	Using Assumption \ref{ass1} and inequality \eqref{lip2}, we obtain:
	\vspace*{-0.2cm}
	\begin{align}
	f(x_{k+1}) + \psi(x_{k+1}) \leq f(x_{k}) + \langle \nabla f(x_{k}), U_{i_k}d_{k} \rangle + \frac{L_{i_{k}}}{2} \|d_{k}\|^2 + \psi(x_{k+1}). \label{eq:68}
	\end{align}
	
	\noindent First, consider $\psi$ convex along coordinates. From optimality condition for  \eqref{eq:subproblem}:
	\begin{align*}
	\langle U^{T}_{i_k} \nabla f(x_{k}) + \Hf d_{k}, d - d_{k} \rangle + \psi (x_{k} + U_{i_k}d)  \geq \psi (x_{k} + U_{i_k}d_{k}) \quad \forall d \in \mathbb{R}^{n_{i}}.
	\end{align*}
	%
	%
	
	\noindent Combining the inequality above for $d = 0$ with  \eqref{eq:68}, using \eqref{eq:H} and \eqref{eq:etaH}, we get:
	\vspace*{-0.1cm}
	\begin{align}
	F(x_{k+1}) 
	&\leq F(x_{k}) +\langle \nabla f(x_{k}), U_{i_k}d_{k} \rangle+ \frac{L_{i_{k}}}{2} \|d_{k}\|^2 - \langle U^{T}_{i_k} \nabla f(x_{k}), d_{k} \rangle  - \Hf \|d_{k}\|^2  \nonumber \\
	&= F(x_{k}) - \left(\Hf - \dfrac{L_{i_{k}}}{2}\right)  \|d_{k}\|^2 \leq  F(x_{k}) - \dfrac{\eta_{\min}}{2}  \|d_{k}\|^2. \label{eq:desc1}
	\end{align}
	
	\noindent Second, consider $\psi$ general function.   Since $d_{k}$ is the optimal solution for \eqref{eq:subproblem}, choosing $d=0$, we have:
	\vspace*{-0.1cm}
	\begin{equation}
	\langle U^{T}_{i_k} \nabla f(x_{k}), d_{k} \rangle + \dfrac{\Hf}{2} \|d_{k}\|^2  + \psi(x_{k} + U_{i_k}d_{k}) \leq  \psi(x_{k}). \label{eq:121}
	\end{equation} 
	
	\noindent From inequalities  \eqref{eq:68} and \eqref{eq:121}, using \eqref{eq:H} and \eqref{eq:etaH}, we also get:
	\vspace*{-0.1cm}
	\begin{eqnarray}
	F(x_{k+1}) &\leq& f(x_{k}) + \langle \nabla f(x_{k}), U_{i_k}d_{k} \rangle + \frac{L_{i_{k}}}{2} \|d_{k}\|^2 + \psi(x_{k+1}). \nonumber\\ 
	&\leq& F(x_{k}) - \dfrac{1}{2}\left(\Hf-L_{i_{k}}\right) \|d_{k}\|^2 \leq  F(x_{k}) - \dfrac{\eta_{\min}}{2}  \|d_{k}\|^2 . \label{eq:desc2}
	\end{eqnarray}	

\end{proof}

\noindent 	Note that the previous lemma  is valid independently on how the index $i_{k}$ is choosen. Moreover,  when $i_{k}$ is choosen uniformly at random  the iterates $x_k$ are random vectors, the function  values $F(x_k)$ are random variables and  $x_{k+1} $ depends on $x_k$ and $i_k$. In the sequel, we assume that the sequence $(x_{k})_{k\geq 0}$ generated by  algorithm CPG is bounded. In Section \ref{Sec:SufCond} we will present sufficient conditions when this holds.  
\red{ 
 Next, we will prove some descent w.r.t. the norm of the gradient. Let us first introduce some notations  that  will be used in the sequel: }
 \vspace*{-0.2cm}
 \begin{equation}
 \red{	\bar{\nabla}^2\Psi(z_{1},\cdots,z_{n})} = \begin{bmatrix}
 		\nabla_{1}^2\psi(z_{1})  \\
 		\vdots \\		
 		\nabla_{n}^2\psi(z_{n}) \label{hess}
 	\end{bmatrix},
 \vspace*{-0.2cm} 
 \end{equation}	
 \noindent with $\nabla_{i}^2\psi(z_{i})$ being the $i$th row of the hessian of $\psi$ at the point $z_{i} \in \mathbb{R}^{n}$,
 \begin{equation}
 	\HFi= \max_{z_{1},\cdots,z_{n} \in \overline{\text{conv}}\{(x_{k})_{k\geq 0}\}} \|U_{i}^{T}\bar{\nabla}^2\Psi(z_{1},\cdots,z_{n})U_{i} + \Hfi I_{n_{i}\times n_{i}} \| \label{Hpsi},
 \end{equation}
\noindent and $\HFmax = \max_{i=1:N} \HFi$.  Note that  CPG algorithm does not require the knowledge of  the coordinate-wise Lipschitz constants of the whole function  $F$ (or of the term  $\psi$). The constant $\HFmax$   only appears in the convergence rates. In \red{some} applications  the second term, $\psi$, although differentiable, might have expensive  gradient evaluation or  the corresponding  coordinate-wise Lipschitz constants over a bounded set might be difficult to estimate; on the other hand,  if the computation of the prox for the second term $\psi$  along a block of coordinates is easy, then  algorithm CPG can be used. \red{ One example is the function $\psi(x) = \|Ax\|^p$, with $p\geq 2$, which has an expensive gradient evaluation when the dimension  of matrix $A$ is very large, since we have to compute a matrix-vector product, and its gradient is Lipschitz on any bounded subset, but the coordinate-wise Lipschitz constants are not easy to estimate.  On the other hand, if we update only one coordinate at each iteration, solving the subproblem \eqref{eq:subproblem} is equivalent to finding a root of a scalar equation. More examples are given in Section 8.}
 For deterministic CPG (i.e.,  cyclic coordinate choice), let us also define $L$: 
 \begin{equation}
 \|U_{i}^{T} \left( \nabla f(x) - 	 \nabla f(y) \right) \| \leq L \| x-y\| \quad \forall  i=1:N \quad \text{and} \quad x,y \in \mathbb{R}^{n},  \label{Lip}
 \end{equation}
 and the constant 
 \begin{equation}
 \Hpsimaxb= \max_{ i=1:N, x \in  \overline{\text{conv}}\{(x_{k})_{k\geq 0}\}} \|U^{T}_{i}\nabla^2 \psi (x)\|. 
 \label{Hpsimax:cyc}  
  \vspace*{-0.2cm}
  \end{equation}
 Since $f$  has coordinate-wise Lipschitz gradient, then there exists  $L>0$ satisfying \eqref{Lip}, see \cite{Nes:10}. 
\begin{lemma} 
	\label{lem2}
	If  Assumption \ref{ass1} holds and the sequence $(x_{k})_{k\geq 0}$ generated by algorithm CPG is bounded, then we have the following descents: \\
	i) If $i_{k}$ is choosen uniformly at random, we have:
	 \vspace*{-0.2cm}
	\begin{align}
		\label{eq:23}
		\mathbb{E}[F(x_{k+1}) \ | \ x_{k}] \leq F(x_{k}) - \dfrac{\eta_{\min}}{2N\HFmax^2} \|\nabla F(x_{k})\|^2 . 
		 \vspace*{-0.2cm}
	\end{align}
	\noindent ii) If $i_{k}$ is choosen cyclic, we have:
	 \vspace*{-0.2cm}
	\begin{align}
		\label{eq:129}
		F(x_{k+N})  \leq F(x_{k}) - \dfrac{\eta_{\min}}{
			4N\Hpsimaxb^2+ 8(N-1)L^2+8\Hfmax^2  } \|\nabla F(x_{k})\|^2 . 
		 \vspace*{-0.2cm}
	\end{align}
\end{lemma}

\begin{proof} 
	\red{i)} First,  we consider $i_{k}$  choosen uniformly at random.  Since $f$, $\psi$ are differentiable, from the  optimality condition for $d_{k}$, we get: 
	\begin{equation}
		\label{oc3}
		U^{T}_{i_k} (\nabla f(x_{k}) + \nabla \psi(x_{k+1})) +\Hf d_{k}  = 0. 
	\end{equation}
	\noindent Moreover,   $\nabla F(x_{k}) = \nabla f(x_{k}) + \nabla \psi(x_{k})$ and 
	\vspace*{-0.1cm}
	\begin{equation*}
		\mathbb{E}[\| U^{T}_{i_{k}} \nabla F(x_{k})\|^2 \ | \ x_{k} ]  = \dfrac{1}{N} \|\nabla F(x_{k})\|^2. 
	\end{equation*}
	\noindent Combining the equality above with \eqref{oc3}, we get:
	\vspace*{-0.1cm}
	\begin{eqnarray}
		\dfrac{1}{N}\|\nabla F(x_{k})\|^{2} &=& \mathbb{E}[\|U^{T}_{i_k}(\nabla f(x_{k}) + \nabla \psi(x_{k})) \|^{2} \ | \ x_{k} ] \nonumber \\
		&=& \mathbb{E}[\|U^{T}_{i_k}(\nabla \psi(x_{k}) - \nabla \psi(x_{k+1})) - \Hf d_{k} \|^{2} \ | \ x_{k} ]. \nonumber  
	\end{eqnarray}
	
\noindent Now, considering the particular form of the matrix $U_{i_k}$, using the mean value theorem and the definition of $H_{F_{i_k}}$, we further get: 
\vspace*{-0.1cm}
\begin{align*}	
& \|U^{T}_{i_k}(\nabla \psi(x_{k}) \!- \nabla \psi(x_{k+1})) -\! \Hf d_{k} \|^{2}   \!=\! \sum_{j \in I_k} |  \nabla_j \psi(x_{k}) -\! \nabla_j \psi(x_{k+1}) - \! \Hf d_{k,j}   |^2 \\
&=   \sum_{j \in I_k} |  \nabla_j^2 \psi(z_{k,j})  U_{i_k} d_k  + \Hf d_{k,j}   |^2   = \|   U_{i_k}^T   	\bar{\nabla}^2\Psi(z_{k,1},\cdots,z_{k,n})  U_{i_k} d_k +  \Hf d_{k}  \|^2 \\
& \leq  \|   U_{i_k}^T   	\bar{\nabla}^2\Psi(z_{k,1},\cdots,z_{k,n})  U_{i_k}  +  \Hf  I_{n_{i_k}}\|^2  \| d_{k}  \|^2  \leq  H_{F_{i_k}} ^2  \| d_{k}  \|^2 \leq  \HFmax^2  \| d_{k}  \|^2,
\end{align*}	

\vspace{-0.0cm}

\noindent where $I_k$ is the set of indexes chosen at  $k$ and $z_{k,j} \in [x_{k}, x_{k+1}] $ for all $j \in I_k$. Hence, we get:
	\begin{equation}
		\|\nabla F(x_{k})\|^{2} \leq N \HFmax^2 \cdot \mathbb{E} \left[ \|d_{k}\|^2 \ | \ x_{k}\right]. \label{eq:128}
	\end{equation}
	\noindent Finally, taking the conditional expectation of both sides of the inequality \eqref{descent} w.r.t. $x_{k}$ and combining it with \eqref{eq:128}, \red{we get \eqref{eq:23}. }  \\
	\red{ii) }If $i_{k}$ is chosen cyclic, then, with some abuse of notation, let us  consider that at the $k$th iteration the first block of coordinates is updated and  at the  $(k+i_{k}-1)$th iteration, we update the $i_{k}$th block of coordinates. Hence, using the optimality condition \eqref{oc3}, we obtain:
	\vspace*{-0.3cm}
	\begin{align}
	&\|\nabla F(x_{k})\|^{2} = \sum_{i_{k} = 1}^{N} \| U_{i_{k}}^{T}\nabla F(x_{k})\|^2 = \sum_{i_{k} = 1}^{N} \| U_{i_{k}}^{T}\left( \nabla f(x_{k})  + \nabla \psi(x_{k})\right)  \|^2 \nonumber \\
	& =   \sum_{i_{k} = 1}^{N} \| U_{i_{k}}^{T}\left( \nabla f(x_{k}) - \nabla f(x_{k+i_{k}-1}) + \nabla \psi(x_{k}) - \nabla \psi(x_{k+i_{k}})\right)  - \Hf d_{k+i_{k} -1} \|^2 \nonumber \\ 
	&\leq \sum_{i_{k} = 1}^{N}\left(  4\| U_{i_{k}}^{T}\left( \nabla f(x_{k}) - \nabla f(x_{k+i_{k}-1}) \right)\|^2  + 4\Hf \|d_{k+i_{k} -1} \|^2\right)  \nonumber \\
	 &+ \sum_{i_{k} = 1}^{N}  2\| U_{i_{k}}^{T}\left( \nabla \psi(x_{k}) - \nabla \psi(x_{k+i_{k}})\right) \|^2. \nonumber
	\end{align}

\vspace*{-0.3cm}
	
 \noindent Note that $\|x_{k+N} - x_{k}\|^2 = \displaystyle \sum_{i_{k}=1}^{N} \|d_{k+i_{k}-1}\|^2$.
	Using  \eqref{eq:etaH}, \eqref{Lip} and the mean value inequality (see Lemma \ref{lemma:MVI})  with $z_{i_{k}} \in [x_{k}, x_{k+i_{k}}]$, we get:
	\vspace*{-0.2cm}
	\begin{align}
		\|\nabla F(x_{k})\|^{2} &\leq 4\Hfmax^2 \|x_{k+N} - x_{k} \|^2 + \sum_{i_{k} = 1}^{N} 4 L^2 \| x_{k} - x_{k+i_{k}-1}\|^2   \nonumber\\
		& +  \sum_{i_{k} = 1}^{N} 2\|	U_{i_{k}}^{T}\nabla^2 \psi(z_{i_{k}})\|^2 \|x_{k+i_{k}} - x_{k} \| ^2.  \nonumber
	\end{align}
	
	\vspace{-0.2cm} 
	
	\noindent \red{ Note that, since one block of coordinates is updated at each iteration, we have $\|x_{k+i_{k}} - x_{k} \|  \leq \|x_{k+N} -x_{k}\|$ for all $i_{k} =1:N-1$.  
	 Hence, from  \eqref{Hpsimax:cyc}, we  obtain:}
	\begin{equation}
		\|\nabla F(x_{k})\|^{2} \leq   \left( 2N\Hpsimaxb^2+  4\Hfmax^2 + 4(N-1)L^2\right)   \|x_{k+N} -x_{k} \|^2. \label{eq:cyc2}
	\end{equation}

\vspace{-0.3cm}

\noindent 	Using  $\|x_{k+N} - x_{k}\|^2=\displaystyle \sum_{i_{k}=1}^{N} \|d_{k+i_{k}-1}\|^2$, from inequality
	\eqref{descent}, we have: 
	\vspace*{-0.5cm}
	\begin{equation}
		F(x_{k+N}) \leq F(x_{k}) - \dfrac{\eta_{\min}}{2} \sum_{i_{k} = 1}^{N} \|d_{k+i_{k} -1}\|^2. \label{eq:cyc3}
		\vspace*{-0.2cm}
	\end{equation}
	\noindent  Finally, combining \eqref{eq:cyc2} and \eqref{eq:cyc3}, we obtain \eqref{eq:129}.   
\end{proof}

\vspace{-0.1cm}

\begin{remark}
	Recall that the convergence analysis  in \cite{BecTet:13}  for cyclic coordinate descent  contains a term $N\bar{L}^2$, with $\bar{L}$ the global  Lipschitz constant of the gradient of $f$, and $\psi=0$. Note that our  $L$ defined in \eqref{Lip} is usually smaller than $\bar{L}$, hence our estimate is usually   better.  
\end{remark}

\section{A coordinate gradient descent algorithm}
\label{sec:RCGD}
In this section, we present a Coordinate Gradient Descent (CGD)  algorithm for solving problem \eqref{eq:prob},  with $f$ and $\psi$ possibly nonseparable and nonconvex. In each iteration, $d_{k}$ is given by some (block) components of the full gradient $\nabla F(x_{k})$.  

\begin{center}
	\noindent\fbox{%
		\parbox{12cm}{%
			\textbf{Algorithm 2 (CGD)}:\\
			Given a starting point $x_{0} \in \rset^n$. \\
			For $k \geq 0$ do: \\
			1. \; Choose $i_{k} \in \{1,...,N\}$ uniformly at random or cyclic and compute $\HF>0$ as defined in one of the following equations: \eqref{eq:HF1}, \eqref{eq:HF2}, \eqref{eq:HF3} or \eqref{eq:HF4}. \\
			2. \; Solve the following subproblem:
			\vspace*{-0.1cm}
			\begin{equation}
			\label{eq:subproblem2}
			\;\; d_{k} = \arg \min_{d \in \mathbb{R}^{n_{i}}} F(x_{k}) +  \langle U^{T}_{i_k} \nabla F(x_{k}), d \rangle + \dfrac{\HF}{2} \|d\|^2.
			\vspace*{-0.1cm}	
			\end{equation} \\
			3. \; 	Update $x_{k+1} = x_{k} + U_{i_k}d_{k}$.
	}}
\end{center}
\medskip
\noindent From the optimality conditions for the subproblem \eqref{eq:subproblem2}, we have:
\vspace*{-0.1cm}
\begin{equation}
d_{k} = - \dfrac{1}{\HF} U^{T}_{i_k} \nabla F(x_{k}) = - \dfrac{1}{\HF} U^{T}_{i_k} \left( \nabla f(x_{k}) + \nabla \psi(x_{k})\right) . \label{eq:d}
\end{equation}
The main difficulty with  algorithm CGD is that we need to find an appropriate stepsize $\HF$ which ensures descent, although the full objective  function $F$ doesn't have a coordinate-wise Lipschitz gradient.  In the sequel we derive  novel stepsize rules which combined with additional properties on $\psi$  yield descent. Let us denote: 
\vspace*{-0.1cm}
\begin{equation}
\Hf = \dfrac{L_{i_{k}} + \eta_{i_{k}}}{2}.  \label{eq:75}
\end{equation}

\noindent Consider one of  the following additional properties on the function $\psi$.
\begin{assumption}
	\label{ass2} \red{ Assume either:}
	\item []A.4: Given function $\psi$,  there exist $\Hp > 0$ and integer $p  \geq 1$ such that:
	\vspace*{-0.1cm}
	\begin{equation*}
	\| U_{i}^{T} \nabla^2 \psi (y) U_{i}\| \leq \Hp \|y\|^{p}  \quad \forall y \in \mathbb{R}^{n}, \quad i = 1:N.  
	\end{equation*}
	\item []A.5: Hessian of $\psi$ is Lipschitz, i.e., there exists $L_{\psi} > 0$ such that:
	\vspace*{-0.1cm}
	\begin{equation*}
	\|\nabla^{2}\psi(y) - \nabla^2\psi(x)\| \leq L_{\psi} \|y-x\| \quad \forall x,y \in \mathbb{R}^{n}.
	\end{equation*}
	\item [] A.6: Function $\psi$ is differentiable and  concave along coordinates, i.e.: 
	 \vspace*{-0.2cm}
	\begin{equation*}
	\psi(x + U_{i}d) \leq \psi(x) + \langle U_{i}^{T}\nabla \psi(x), d\rangle \quad \forall d \in \mathbb{R}^{n_{i}}, x\in\mathbb{R}^{n}, i =1:N.
	\end{equation*}
	
\end{assumption}

\noindent See  Section \ref{Simulations} for concrete examples of functions satisfying Assumption \ref{ass2} [A.4-A.5].   

\noindent For simplicity of the exposition, in Table \ref{table:stepsizes} we present four stepsize rules  and the corresponding assumptions  on $\psi$ which allows us to prove   descent for algorithm CGD. Note that, in order to run  algorithm CGD, we need to know  $\Hp$ or $L_{\psi}$, respectively, and  the third stepsize strategy requires computation of $\nabla^2 \psi$ only in $x_{0}$.
	
\footnotesize{
\begin{longtable}{| c| c | c|}
	\hline
	 &  Stepsize choice& Ass. on $\psi$ \\
	\hline
	 & \multirow{11}{9.5cm}{ 1.   Choose $\Hf > \dfrac{L_{i_{k}}}{2}$ and compute $\alpha_{k}\geq0$ as root of second order equation in $\alpha$: 
	\begin{equation}
	\dfrac{L_{\psi}}{6} \alpha^{2} + \left(  \dfrac{\Hp}{2} \|x_{k}\|^{p} + \Hf\right) \alpha - \|U_{i_{k}}^{T} \nabla F(x_{k})\| = 0. \label{eq:103}
	\end{equation}
	2. Define
	\vspace*{-0.1cm}
	\begin{equation}
	\HF = \frac{\Hp}{2}\|x_{k}\|^p + \frac{L_{\psi}}{6} \alpha_{k} + \Hf. \label{eq:HF1}
	\vspace*{-0.2cm}
	\end{equation}} & \\ 
	& & \\
	& & \\
	& & \\
	& & A.4  \\
	1) & &  and\\
	& & A.5 \\
	& & \\
	& & \\
	& & \\\hline
	& \multirow{10}{9.5cm}{ 1. Choose $\Hf > \dfrac{L_{i_{k}}}{2}$ and compute $\alpha_{k}\geq0$ as root of the following polynomial equation in $\alpha$:
		\begin{equation}
		2^{p-1} \Hp \alpha^{p+1} + (2^{p-1}  \Hp \|x_{k}\|^{p} + \Hf)\alpha 
		= \|U_{i_{k}}^{T} \nabla F(x_{k})\|. \label{eq:100}
		\end{equation}
		2.  Define:
		\vspace*{-0.1cm}
		\begin{equation}
		\HF = 2^{p-1} \Hp\|x_{k}\|^p + 2^{p-1} \Hp \alpha_{k}^p + \Hf. \label{eq:HF2}
		\vspace*{-0.3cm}
		\end{equation}
} & \\ 
	& & \\
	& & \\
	& & \\
	& & \\
	2) & &  A.4 \\
	& & \\
	& & \\
	& & \\  \hline
	& \multirow{12}{9.5cm}{ 1. \;  Choose $\Hf > \dfrac{L_{i_{k}}}{2}$ and compute $\alpha_{k}\geq0$ as root of the second order equation:
		\begin{eqnarray}
		\dfrac{L_{\psi}}{6} \alpha^{2} + \left(  \dfrac{L_{\psi}}{2} \|x_{k}-x_{0}\| + \dfrac{1}{2} \|\nabla^2 \psi(x_{0})\| + \Hf\right) \alpha \nonumber \\
		 = \|U_{i_{k}}^{T} \nabla F(x_{k})\|. \label{eq:105}
		\end{eqnarray}
		2. \; Update 
		\begin{equation}
		\vspace*{-0.1cm}
		\HF = \frac{L_{\psi}}{2}\|x_{k}-x_{0}\| + \frac{1}{2}\|\nabla^2 \psi(x_{0})\| + \frac{L_{\psi}}{6} \alpha_{k} + \Hf \label{eq:HF3}
		\vspace*{-0.3cm}
		\end{equation}} & \\ 
	& & \\
	& & \\
	& & \\
	& & \\
	& & \\
	3) & &  A.5 \\
	& & \\
	& & \\
	& & \\
	& & \\
	& & \\ \hline
	& \multirow{5.5}{8cm}{1. \;  Choose $\Hf > \dfrac{L_{i_{k}}}{2}$ and update 
	\vspace*{-0.1cm}
	\begin{equation}
	\HF = \Hf. \label{eq:HF4}
	\vspace*{-0.1cm}
	\end{equation}} & \\
	4)& & A.6\\
	& &  \\
	& & \\ \hline
	\caption{Proposed stepsize rules for the algorithm CGD. }
	\label{table:stepsizes}
\end{longtable}
}
\normalsize
\vspace*{-0.5cm}
\noindent Note that, the first three stepsize rules are adaptive and require at each iteration computation of  a nonnegative root of  some polynomial,  while the last one is chosen constant. Moreover,  the case 4) covers difference of convex (DC) programming problems and our algorithm CGD is new in this context.  One can easily see that all the equations \eqref{eq:103}, \eqref{eq:100} and  \eqref{eq:105} admit a nonnegative root $\alpha_{k} \geq 0$ and thus $\HF$ is well-defined. Indeed, let us check for the second stepsize choice. Consider:
\vspace*{-0.01cm} 
\begin{equation}
	h(\alpha) = 2^{p-1} \Hp \alpha^{p+1} + (2^{p-1}  \Hp \|x_{k}\|^{p} + \Hf)\alpha - \|U_{i_{k}}^{T} \nabla F(x_{k})\| \label{eq:h}
	\vspace*{-0.1cm}
\end{equation}

\noindent and $w_{k} = \dfrac{1}{\Hf} \|U_{i_{k}}^{T} \nabla F(x_{k})\|$. If $\|U_{i_{k}}^{T} \nabla F(x_{k})\| \neq 0$, then we have $h(w_{k}) > 0$ and $h(0) < 0$.  Since $h$ is continuous on $[0,w_{k}]$,  there exists $\alpha_{k} \in (0,w_{k})$ such that $h(\alpha_{k}) = 0$. Moreover, since  $h'(\alpha) > 0$ for all $\alpha \in (0,+\infty)$, then $h$ is strictly increasing on $(0,+\infty)$. Hence, there exists exactly one $\alpha_{k} > 0$ satisfying \eqref{eq:100}. Otherwise, if $\|U_{i_{k}}^{T} \nabla F(x_{k})\| = 0$, we have $\alpha_{k} = 0$.   One can see that the first three stepsizes  satisfy:
\vspace*{-0.1cm}
\begin{equation}
	\|d_{k}\| = \dfrac{1}{\HF}\|U_{i_{k}}^{T} \nabla F(x_{k})\|  = \alpha_{k}. \label{eq:104}
	\vspace*{-0.1cm}
\end{equation}

\begin{lemma}
	\label{lem:desRCGD2}
	Let   Assumptions \ref{ass1} and  \ref{ass2} hold such that $\HF$ is updated according to Table \ref{table:stepsizes}.  Then,  the iterates of algorithm CGD satisfy the descent:
		\vspace*{-0.1cm}
	\begin{align}
	\label{descent2}	
     F(x_{k+1})  \leq F(x_{k}) - \dfrac{\eta_{\min}}{2} \|d_{k}\|^2 . 
     \vspace*{-0.1cm}
	\end{align}
\end{lemma}

\begin{proof}
\vspace*{-0.1cm}
Consider first case	\textit{1)}, i.e.,  conditions  A.4 and A.5 of Assumption \ref{ass2}  hold.  From Assumption \ref{ass2}[A.5], we have:
		\vspace*{-0.2cm}
	\begin{equation*}
	\psi(x_{k+1}) \leq \psi(x_{k}) + \langle \nabla \psi(x_{k}),U_{i_{k}}d_{k}\rangle + \dfrac{1}{2}\langle \nabla^2 \psi(x_{k})U_{i_{k}}d_{k},U_{i_{k}}d_{k}\rangle + \dfrac{L_{\psi}}{6} \|d_{k}\|^3.
	\end{equation*} 
	
	\noindent Combining the previous  inequality  with \eqref{lip2}, we obtain:
	\vspace*{-0.1cm}
	\begin{eqnarray*}
		&&f(x_{k+1}) + \psi(x_{k+1}) \leq f(x_{k}) + \langle \nabla f(x_{k}), U_{i_k}d_{k} \rangle + \frac{L_{i_{k}}}{2} \|d_{k}\|^2 + \psi(x_{k+1}) \\  
		&&\leq f(x_{k}) + \langle U_{i_{k}}^{T}\nabla F(x_{k}), d_{k} \rangle + \frac{L_{i_{k}}}{2} \|d_{k}\|^2 + \psi(x_{k})  + \dfrac{1}{2}\langle U_{i_{k}}^{T}\nabla^2 \psi(x_{k})U_{i_{k}}d_{k},d_{k}\rangle + \dfrac{L_{\psi}}{6} \|d_{k}\|^3. 
	\end{eqnarray*}	
	
	\noindent Further, from \eqref{eq:d}, we have:
		\vspace*{-0.1cm}
	\begin{align}
		\label{eq:107}  
	& F(x_{k+1}) \\
	& \leq F(x_{k}) - \HF\|d_{k}\|^{2} + \frac{L_{i_{k}}}{2} \|d_{k}\|^2  + \dfrac{1}{2}\langle U_{i_{k}}^{T}\nabla^2 \psi(x_{k})U_{i_{k}}d_{k},d_{k}\rangle + \dfrac{L_{\psi}}{6} \|d_{k}\|^3   \nonumber\\
	& \leq F(x_{k}) - \HF\|d_{k}\|^{2} + \frac{L_{i_{k}}}{2} \|d_{k}\|^2  + \dfrac{1}{2} \|U_{i_{k}}^{T}\nabla^2 \psi(x_{k})U_{i_{k}}\| \|d_{k}\|^2 + \dfrac{L_{\psi}}{6} \|d_{k}\|^3. \nonumber
	\end{align}
	
	\noindent From  Assumption \ref{ass2}[A.4], we obtain:
		\vspace*{-0.1cm}
	\begin{equation*}
	F(x_{k+1}) \leq F(x_{k}) - \HF \|d_{k}\|^{2} + \frac{L_{i_{k}}}{2} \|d_{k}\|^2  + \dfrac{\Hp}{2} \| x_{k}\|^p \|d_{k}\|^2 + \dfrac{L_{\psi}}{6} \|d_{k}\|^3.
	\vspace*{-0.1cm}
	\end{equation*}  
	
	\noindent From \eqref{eq:HF1} and \eqref{eq:104}, we have $\HF = \frac{L_{\psi}}{6} \|d_{k}\| + \frac{\Hp}{2} \|x_{k}\|^{p} + \Hf$. Then, from \eqref{eq:75} and \eqref{eq:etaH}, we get the descent:
		\vspace*{-0.2cm}
	\begin{equation}
	F(x_{k+1}) \leq F(x_{k}) - \left( \Hf - \frac{L_{i_{k}}}{2}\right)  \|d_{k}\|^2 \leq F(x_{k}) - \frac{\eta_{\min}}{2} \|d_{k}\|^2. \label{eq:122}
	 \vspace*{-0.1cm}
	\end{equation}
	
\noindent Consider now case	\textit{2)}, i.e.,  A.4 of Assumption \ref{ass2} holds. Since $\psi$ is differentiable, from the mean value theorem there exists $y_k \in [x_{k},x_{k} + U_{i_k}d_{k}]$ such that
	$\psi(x_{k+1}) - \psi(x_{k}) = \langle \nabla \psi(y_{k}), U_{i_{k}}d_{k}\rangle$.
	Combining the last equality with \eqref{lip2}, we obtain:
	\vspace*{-0.2cm}
	\begin{eqnarray*}
		f(x_{k+1}) + \psi(x_{k+1}) &\leq& f(x_{k}) + \langle \nabla f(x_{k}), U_{i_k}d_{k} \rangle + \frac{L_{i_{k}}}{2} \|d_{k}\|^2 + \psi(x_{k+1}) \\  
		&=& f(x_{k}) + \psi(x_{k}) + \langle \nabla \psi(y_{k}) + \nabla f(x_{k}), U_{i_{k}}d_{k}\rangle + \frac{L_{i_{k}}}{2} \|d_{k}\|^2. 
		\vspace*{-0.1cm}
	\end{eqnarray*}	
	
	\noindent Using \eqref{eq:d}, we further have:
	\vspace*{-0.1cm}
	\begin{equation}
		F(x_{k+1}) \leq F(x_{k}) + \langle U_{i_{k}}^{T}\left(  \nabla \psi(y_{k}) - \nabla \psi(x_{k})\right) , d_{k}\rangle - \HF \|d_{k}\|^2 + \frac{L_{i_{k}}}{2} \|d_{k}\|^2. \label{eq:101}
	\end{equation}  
	
	\noindent Since $y_k \in [x_{k},x_{k}+U_{i_k}d_{k}]$, then $y_{k} = (1-\tau)x_{k} + \tau\left( x_{k} + U_{i_k}d_{k}\right)$ for some $\tau \in [0,1]$. Moreover,  from Lemma \ref{lemma:MVI} there exists $\bar{x}_k \in [x_{k},y_{k}]$ such that: 
	$$\|U^{T}_{i_k}(\nabla \psi(y_{k}) - \nabla \psi(x_{k}))\|  \leq \|U^{T}_{i_k}\nabla^2 \psi (\bar{x}_k) U_{i_{k}} \|   \|\tau d_{k}\| \leq \|U^{T}_{i_k}\nabla^2 \psi (\bar{x}_k) U_{i_{k}} \|   \| d_{k}\|.$$
	Note that $\bar{x}_k =  (1-\mu) x_{k} + \mu y_{k} $ for some $\mu \in [0,1]$. From  Assumption \ref{ass2}[A.4] and the last inequality, we obtain:
	\vspace*{-0.1cm}
	\begin{eqnarray*}
	\langle U_{i_{k}}^{T}\left(  \nabla \psi(y_{k}) - \nabla \psi(x_{k})\right) , d_{k}\rangle &\leq&      
	\|U_{i_{k}}^{T}\left(  \nabla \psi(y_{k}) - \nabla \psi(x_{k})\right) \| \|d_{k}\| \\ 
	\leq \|U^{T}_{i_k}\nabla^2 \psi (\bar{x}_k) U_{i_k}\|  \|d_{k}\|^2 
	&\leq& \Hp \|(1-\mu) x_{k} + \mu y_{k}\|^{p}  \|d_{k}\|^2.
\end{eqnarray*}	

\noindent From convexity of $\|\cdot\|^{p}$, for $p \geq 1$,  and the fact that $y_{k} = (1-\tau)x_{k} + \tau\left( x_{k} + U_{i_k}d_{k}\right)$ for some $\tau \in [0,1]$, we get:
	\vspace*{-0.1cm}
	\begin{eqnarray*}
		&& \langle U_{i_{k}}^{T}\left(\nabla \psi(y_{k}) - \nabla \psi(x_{k})\right) , d_{k}\rangle  \leq \Hp \left( (1-\mu) \|x_{k}\|^p + \mu\|y_{k}\|^{p}\right) \|d_{k}\|^2 \\
		&& = \Hp \left( (1-\mu) \|x_{k}\|^p + \mu\|(1-\tau)x_{k} + \tau\left( x_{k} + U_{i_k}d_{k}\right)\|^{p}\right) \|d_{k}\|^2 \\
		&& \leq \Hp \left( (1-\mu) \|x_{k}\|^p + \mu(1-\tau)\|x_{k}\|^p + \mu \tau \| x_{k} + U_{i_k}d_{k}\|^{p}\right) \|d_{k}\|^2.
	\end{eqnarray*}
	
	\noindent Since $\mu,\tau \in[0,1]$ and $\|a + b\|^p \leq 2^{p-1} \|a\|^p + 2^{p-1} \|b\|^p$ for $p \geq 1$, we get:
	\begin{eqnarray*}
		\langle U_{i_{k}}^{T}\left(\nabla \psi(y_{k}) - \nabla \psi(x_{k})\right) , d_{k}\rangle &\leq& \Hp \left( (1+(2^{p-1} -1)\mu \tau) \|x_{k}\|^p + 2^{p-1} \mu \tau \|d_{k}\|^{p}\right) \|d_{k}\|^2 
		\\ &\leq& 2^{p-1} \Hp \|x_{k}\|^{p}\|d_{k}\|^2 + 2^{p-1} \Hp \|d_{k}\|^{p+2}.
	\end{eqnarray*}
	
	\noindent Combining the inequality above with \eqref{eq:101}, we obtain:
	\vspace*{-0.2cm}
	\begin{equation*}
		F(x_{k+1}) \leq F(x_{k}) + 2^{p-1} \Hp \|x_{k}\|^{p}\|d_{k}\|^2 + 2^{p-1} \Hp \|d_{k}\|^{p+2} - \HF \|d_{k}\|^2 + \frac{L_{i_{k}}}{2} \|d_{k}\|^2. 
		\vspace*{-0.0cm}
	\end{equation*}
	
	\noindent From \eqref{eq:HF2} and \eqref{eq:104}, we have $\HF = 2^{p-1} \Hp \|d_{k}\|^{p} + 2^{p-1} \Hp\|x_{k}\|^{p} + \Hf$. Hence,  from \eqref{eq:75} and \eqref{eq:etaH}, we get the descent \eqref{eq:122}.
	
\noindent Consider now case	\textit{3)}, i.e.,  A.5 of Assumption \ref{ass2} holds.  From Assumption \ref{ass1}, Assumption \ref{ass2}[A.5] and  \eqref{eq:107}, we have:
	\vspace*{-0.1cm}
	\begin{eqnarray*}
		F(x_{k+1}) &\leq& F(x_{k}) - \HF \|d_{k}\|^{2} + \frac{L_{i_{k}}}{2} \|d_{k}\|^2  + \dfrac{1}{2}\langle \nabla^2 \psi(x_{k})U_{i_{k}}d_{k},U_{i_{k}}d_{k}\rangle + \dfrac{L_{\psi}}{6} \|d_{k}\|^3 \\
		&\leq& F(x_{k}) - \HF \|d_{k}\|^{2} + \frac{L_{i_{k}}}{2} \|d_{k}\|^2  + \dfrac{1}{2}\|\nabla^2 \psi(x_{k})\| \|d_{k}\|^2 + \dfrac{L_{\psi}}{6} \|d_{k}\|^3 \\
		&\leq& F(x_{k}) - \HF \|d_{k}\|^{2} + \frac{L_{i_{k}}}{2} \|d_{k}\|^2 + \dfrac{L_{\psi}}{6} \|d_{k}\|^3 \\
		&& + \dfrac{1}{2}\left(\|\nabla^2 \psi(x_{k})-\nabla^2 \psi(x_{0})\| + \|\nabla^2 \psi(x_{0})\|\right)  \|d_{k}\|^2 \\
		&\leq& F(x_{k}) - \HF\|d_{k}\|^{2} + \frac{L_{i_{k}}}{2} \|d_{k}\|^2 + \dfrac{L_{\psi}}{6} \|d_{k}\|^3 \\
		&& + \dfrac{1}{2}\left(L_{\psi}\|x_{k}-x_{0}\| + \|\nabla^2 \psi(x_{0})\|\right)  \|d_{k}\|^2.
	\end{eqnarray*}
	
	\noindent From \eqref{eq:HF3} and \eqref{eq:104}, we have $\HF = \frac{L_{\psi}}{6} \|d_{k}\| + \frac{L_{\psi}}{2} \|x_{k}-x_{0}\| + \frac{1}{2} \|\nabla^2 \psi(x_{0})\| + \Hf$. Hence, by \eqref{eq:75} and \eqref{eq:etaH} we get \eqref{eq:122}.

	\noindent Finally, consider the case \textit{4)}, i.e., A.6 of Assumption \ref{ass2} holds. Since $\psi$ is concave along coordinates, we have:
	\vspace*{-0.3cm}
	\begin{equation*}
		\psi(x_{k+1}) \leq \psi(x_{k}) + \langle U_{i_{k}}^{T}\nabla \psi(x_{k}), d_{k}\rangle.
	\end{equation*}
	
	\noindent Combining the inequality above with \eqref{lip2}, we obtain:
	\vspace*{-0.1cm}
	\begin{eqnarray}
		&& f(x_{k+1}) + \psi(x_{k+1}) \nonumber \\ 
		&& \leq f(x_{k}) + \langle \nabla f(x_{k}), U_{i_k}d_{k} \rangle + \frac{L_{i_{k}}}{2} \|d_{k}\|^2 + \psi(x_{k}) + \langle U_{i_{k}}^{T}\nabla \psi(x_{k}), d_{k}\rangle \nonumber \\  
		&& \leq F(x_{k}) + \langle U_{i_{k}}^{T}\nabla F(x_{k}), d_{k} \rangle + \frac{L_{i_{k}}}{2} \|d_{k}\|^2.  \label{eq:120}
	\end{eqnarray}	
	
	\noindent From \eqref{eq:d} and \eqref{eq:HF4}, we have $\Hf d_{k} = -U_{i_{k}}^{T}\nabla F(x_{k})$. Hence we obtain \eqref{eq:122}.	
\end{proof}

\noindent 	Note that previous lemma is valid independently of the way $i_{k}$ is choosen.  Further, from \eqref{descent2}, we have: 
\vspace*{-0.2cm}
\begin{equation*}
\frac{\eta_{\min}}{2} \sum_{j = 0}^{k} \|d_{j}\|^2
\leq \sum_{j = 0}^{k}  F(x_{j}) - F(x_{j+1}) \leq  F(x_{0}) - F^{*} < \infty.
\vspace*{-0.1cm}
\end{equation*}

\noindent with $F^{*}$ defined in \eqref{eq:prob}. This implies that $\|d_{j}\| \to 0$ as $j\to \infty$ in all four cases. Hence, there exists $B_{1} > 0$ such that:
\vspace*{-0.05cm}
\begin{equation}
\|d_{k}\| \leq B_{1} \quad \forall k \geq 0.  \label{eq:110}
\vspace*{-0.05cm}
\end{equation}

\noindent In order to prove  next lemma, we assume that the sequence $(x_{k})_{k\geq 0}$ generated by RCGD algorithm is bounded, i.e., there exists $B_{2} > 0$ such that:
\vspace*{-0.05cm}
\begin{equation}
\|x_{k}\| \leq B_{2} \quad  \forall k \geq 0.  \label{eq:111}
\vspace*{-0.05cm}
\end{equation}

\noindent In Section \ref{Sec:SufCond} we derive sufficient conditions for \eqref{eq:111} to hold. Let us define: 
\begin{equation}
	\HFmaxc= \max_{ \{i=1:N, x \in \overline{\text{conv}}\{(x_{k})_{k\geq 0}\}} \|U^{T}_{i}\nabla^2 F (x)\| < \infty, \label{HFmax:cyc}  
\end{equation}

\noindent that is bounded since we assume $(x_{k})_{k\geq 0}$ bounded. For simplicity, consider for the random variant:
\begin{equation*}
	\footnotesize{
	C_1 = \left\lbrace\begin{array}{ll} \dfrac{\eta_{\min}}{2N (\frac{\Hp}{2}B_2^p + \frac{L_{\psi}}{6} B_1 + \Hfmax)^2}, 
		\text{ if A.4 and A.5 hold} \\
		\dfrac{\eta_{\min}}{2N (2^{p-1}\Hp B_1^{p} + 2^{p-1}\Hp B_2^{p} + \Hfmax)^2}, \text{ if  A.4 holds} \\
		\dfrac{\eta_{\min}}{2N \left(L_{\psi}\left( \frac{B_1}{6} + \frac{B_2}{2} + \frac{1}{2}\|x_{0}\| \right)  + \frac{1}{2}\|\nabla^2 \psi(x_{0})\| + \Hfmax\right)^2 }, \text{if A.5 holds } \\
		\dfrac{\eta_{\min}}{2N\Hfmax^2}, \text{if A.6 holds}
	\end{array}\right.  }
\end{equation*}
and for the cyclic variant: 
\begin{equation*}
	\footnotesize{
	C_2 = \left\lbrace\begin{array}{ll} \dfrac{\eta_{\min}}{4 (\frac{\Hp}{2}B_2^p + \frac{L_{\psi}}{6} B_1 + \Hfmax)^2 + 4N\HFmaxc^2}, 
		\text{ if A.4 and A.5 hold} \\
		\dfrac{\eta_{\min}}{4 (2^{p-1}\Hp B_1^{p} + 2^{p-1}\Hp B_2^{p} + \Hfmax)^2 + 4(N-1)\HFmaxc^2}, \text{ if  A.4 holds} \\
		\dfrac{\eta_{\min}}{4 \left(L_{\psi}\left( \frac{B_1}{6} + \frac{B_2}{2} + \frac{1}{2}\|x_{0}\| \right)  + \frac{1}{2}\|\nabla^2 \psi(x_{0})\| + \Hfmax\right)^2 + 4(N-1)\HFmaxc^2}, \text{if A.5 holds} \\
		\dfrac{\eta_{\min}}{4\Hfmax^2 + 4(N-1)\HFmaxc^2} \text{if A.6 holds}.
	\end{array}\right.  }
\end{equation*}

\begin{lemma}
\label{lem:RCGD2}
Let  assumptions of Lemma \ref{lem:desRCGD2} hold. Additionally, let  the sequence $(x_{k})_{k\geq 0}$ generated by algorithm CGD be bounded. Then, the following descents hold: \\
\vspace*{-0.1cm}
i) If $i_{k}$ is choosen uniformly at random, we have:
\begin{equation}
	\mathbb{E}[F(x_{k+1}) \ | \ x_{k}] \leq F(x_{k}) - C_{1} \|\nabla F(x_{k})\|^2 . \label{eq:126}
\end{equation}

\noindent ii) If $i_{k}$ is choosen cyclic, we have:
\begin{equation}
	F(x_{k+N})  \leq F(x_{k}) - C_{2} \|\nabla F(x_{k})\|^2.  \label{eq:131}
\end{equation}
\end{lemma}

\begin{proof}
	Consider the case 1), i.e., when A.4 and A.5 of Assumption \ref{ass2} hold.  First we analyse the case when  $i_{k}$ is updated uniformly at random. Then, taking the expectation on both sides of the inequality \eqref{descent2} w.r.t. $x_{k}$  and using \eqref{eq:d}, we have: 
	\vspace*{-0.1cm}
	\begin{equation*}
		\mathbb{E}[F(x_{k+1}) \ | \ x_{k}] \leq F(x_{k}) - \dfrac{\eta_{\min}}{2} \mathbb{E}\left[\frac{1}{\HF^2}\|U_{i_{k}}^{T} \nabla F(x_{k})\|^2 \ | \ x_{k}\right]. 
		\vspace*{-0.1cm}
	\end{equation*}
	
	\noindent Combining \eqref{eq:104} and \eqref{eq:110}, we get that $\alpha_{k} = \|d_k\| \leq B_{1}$. Further, from \eqref{eq:HF1}, \eqref{eq:111} and \eqref{eq:etaH}, we have
	$ \HF \leq \dfrac{\Hp}{2}B_2^p + \dfrac{L_{\psi}}{6} B_1 + \Hfmax. $
	This implies that:
	\begin{equation*}
		\mathbb{E}[F(x_{k+1}) \ | \ x_{k}] \leq F(x_{k}) - \dfrac{\eta_{\min} \cdot \mathbb{E}\left[\|U_{i_{k}}^{T} \nabla F(x_{k})\|^2 \ | \ x_{k}\right]}{2\left( \frac{\Hp}{2}B_2^p + \frac{L_{\psi}}{6} B_1 + \Hfmax\right)^2}. 
	\end{equation*}
	\noindent Since
	$ \mathbb{E}[\| U^{T}_{i_{k}} \nabla F(x_{k})\|^2 \ | \ x_{k} ]  = \dfrac{1}{N} \|\nabla F(x_{k})\|^2$, 
	the statement follows. Note that \eqref{eq:126} can be proved  similarly for the other choices of the stepsize.  Second, let us analyse the case when  $i_{k}$ is updated cyclic. For simplicity, consider the $k$th iteration such that the first block of coordinates is updated. Then: 
	\vspace*{-0.2cm}
	\begin{align*}
	\|\nabla F(x_{k})\|^{2} &= \sum_{i_{k} = 1}^{N} \| U_{i_{k}}^{T}\nabla F(x_{k})\|^2 \\
	&= \sum_{i_{k} = 1}^{N} \| U_{i_{k}}^{T}\left( \nabla F(x_{k}) - \nabla F(x_{k + i_{k} - 1   })  + \nabla F(x_{k + i_{k} -1 })\right) \|^2.
	\end{align*}
	\noindent Note that in the $(k+i_{k}-1)$th iteration, we update the $i_{k}$th block of coordinates.  Hence, using \eqref{eq:d}, we have:
	\begin{align}
		\|\nabla F(x_{k})\|^{2} 
		&= \sum_{i_{k} = 1}^{N} \| U_{i_{k}}^{T}\left( \nabla F(x_{k}) - \nabla F(x_{k + i_{k} - 1   }) \right) + H_{F_{k + i_{k} -1}} d_{k + i_{k} -1}   \|^2 \nonumber \\ 
		& \leq \sum_{i_{k} = 1}^{N} 2 \left(  \| U_{i_{k}}^{T}\left( \nabla F(x_{k}) - \nabla F(x_{k + i_{k} - 1   }) \right\|^2 + H_{F_{k + i_{k} -1}}^2 \|d_{k + i_{k} -1}   \|^2\right). \label{eq:cyc4}
	\end{align}
	
	\noindent Using Lemma \ref{lemma:MVI}, we have that there exists $z_{i_{k}} \in [x_{k}, x_{k+i_{k}}]$ such that: 
	\begin{equation*}
		\| U_{i_{k}}^{T}\left( \nabla F(x_{k}) - \nabla F(x_{k + i_{k} - 1   }) \right\| \leq \|U_{i_{k}}^{T}\nabla^2 F(z_{i_{k}}) \| \|x_{k} - x_{k + i_{k} - 1} \|.
	\end{equation*}

	\noindent  Recall that $ \HF \leq \dfrac{\Hp}{2}B_2^p + \dfrac{L_{\psi}}{6} B_1 + \Hfmax$ for all $k \geq 0$. Using $\|x_{k+N} - x_{k}\|^2 = \displaystyle \sum_{i_{k}=1}^{N} \|d_{k+i_{k}-1}\|^2$, \eqref{HFmax:cyc} and \eqref{eq:cyc4}, we obtain: 
	\begin{align*}
		\|\nabla F(x_{k})\|^{2}
		& \leq 2(N-1)\HFmaxc + 2\left( \dfrac{\Hp}{2}B_2^p + \dfrac{L_{\psi}}{6} B_1 + \Hfmax\right) ^2 \|x_{k+N}-x_{k} \|^2. \label{eq:cyc4}
	\end{align*}
	
	\noindent Moreover, by inequality \eqref{descent2}, we get
$
		F(x_{k+N}) \leq F(x_{k}) - \dfrac{\eta_{\min}}{2} \sum_{i_{k} = 1}^{N} \|d_{k+i_{k} -1}\|^2. 
$

	\noindent Combining the last two inequalities we get the statement. Note that, the other cases can be proved similarly. 
\end{proof}

\noindent Note that in the case 4), i.e., when A.6 of Assumption \ref{ass2} holds, we don't need to require that the sequence $(x_{k})_{k\geq 0}$ is bounded.  Moreover, in this case $\psi$  can be only once differentiable.  Next,  we provide sufficient conditions when the sequences generated by our two random coordinate descent algorithms are bounded.


\section{Sufficient conditions for bounded iterates}
\noindent Note that  Lemmas \ref{lem1} and  \ref{lem:desRCGD2} prove that the sequence  $\left(F(x_{k})\right)_{k\geq 0}$ generated by the algorithms CPG or CGD (with appropiate stepsize rules) is nonincreasing, i.e., $F(x_{k+1}) \leq F(x_{k})$ for all $k \geq 0$.  However, in order to prove Lemmas \ref{lem2} and \ref{lem:RCGD2}, we need to assume that the sequence $(x_{k})_{k\geq0}$ is bounded. In this section we present sufficient conditions for boundedness. One natural example  is when the level set is bounded: 
\vspace*{-0.2cm}
\begin{equation*}
\mathcal{L}_{F}(x_{0})=\left\{x:  \,F(x)\leq F(x_{0}) \right\}.
 \vspace*{-0.2cm}
\end{equation*} 
Note that uniformly convex functions have bounded level sets. Indeed, if $F$ is uniformly convex, with constant $q>1$, then it satisfies \cite{Nes:19}: 
\vspace*{-0.1cm}
\begin{equation}
F(y) \geq F(x) + \langle \nabla F(x), y - x \rangle + \frac{\sigma_q (q-1)}{q} \|y - x\|^{\frac{q}{q-1}} \quad \forall x, y \in \mathbb{R}^{n}. \label{eq:unifcon}
\vspace*{-0.1cm}
\end{equation}

\noindent Then, for $x = x^{*}$ and $y \in \mathcal{L}_{F}(x_{0})$, we have
\vspace*{-0.1cm}
\begin{equation*}
\frac{\sigma_q (q-1)}{q} \|x^{*} - y\|^{\frac{q}{q-1}} \leq F(y) - F(x^{*}) \leq F(x^0) - F(x^{*}) < \infty.
\vspace*{-0.1cm}
\end{equation*}

\noindent Moreover, note that if $F$ is uniformly convex, then it has a unique minimizer. In  the  next lemma we show that if $f$ is nonconvex and $\psi$ uniformly convex with constant $q\in (1,2)$, then the level set  $\mathcal{L}_{F}(x_{0})$ is bounded.

\begin{lemma}
	Let $\mathcal{X}^{*}$ be the set of optimal solutions of problem \eqref{eq:prob}. Assume  $\psi$ differentiable and uniformly convex function  with constant $q\in (1,2)$ and the function $f$ satisfies Assumption \ref{ass1}[A.1]. Then:
	\vspace*{-0.1cm} 
	\begin{equation}
	\|x -x^{*}\| \leq \max \left\lbrace \! \left(\dfrac{(2(F(x_{0}) - F(x^{*}))+ L) q}{2(q-1)\sigma_{q}}\right)^{\frac{q-1}{2-q}}\!\!, 1      \right\rbrace \\ \;\; \forall x \in \mathcal{L}_{F}(x_{0}),  x^{*} \in \mathcal{X}^{*}.  \label{eq:bound}
	\end{equation}
\noindent Moreover, if $\mathcal{X}^{*}$ is bounded,  then the level set $\mathcal{L}_{F}(x_{0})$ is also bounded
\end{lemma}

\begin{proof}
	We prove \eqref{eq:bound} by contradiction. Assume for some $\bar{x} \in \mathcal{L}_{F}(x_{0})$ and $x^{*} \in \mathcal{X}^{*}$ that
	\vspace*{-0.1cm} 
	\begin{equation}
	\|\bar{x} -x^{*}\| > \max \left\lbrace \left(\dfrac{(2(F(x_{0}) - F(x^{*}))+ L) q}{2(q-1)\sigma_{q}}\right)^{\frac{q-1}{2-q}}, 1      \right\rbrace. \label{eq:114}
	\end{equation}
	
	\noindent Then, 
	$  \|\bar{x} -x^{*}\|^{\frac{2-q}{q-1}} > \dfrac{(2(F(x_{0}) - F(x^{*}))+ L) q}{2(q-1)\sigma_{q}}, $
	or equivalently: 
	\vspace*{-0.1cm} 
	\begin{equation}
	\dfrac{(q-1)\sigma_{q}}{q}  \|\bar{x} -x^{*}\|^{\frac{2-q}{q-1}} > \dfrac{L}{2} + F(x_{0}) - F(x^{*}) \geq \dfrac{L}{2}. \label{eq:93}
	\end{equation}
	
	\noindent Since $\psi$ is uniformly convex, we have:
	\vspace*{-0.1cm} 
	\begin{equation*}
	\psi(x) \geq \psi(y) + \langle \nabla \psi(y), x - y \rangle + \frac{\sigma_q (q-1)}{q} \|y - x\|^{\frac{q}{q-1}} \quad \forall x, y \in \mathbb{R}^{n}. 
	\vspace*{-0.1cm} 
	\end{equation*}
	
	\noindent Taking $y=x^{*}$ and $x=\bar{x}$, we get:
	\vspace*{-0.1cm} 	
	\begin{equation}
	\dfrac{\sigma_q (q-1)}{q} \|\bar{x} - x^{*}\|^{\frac{q}{q-1}} \leq \psi(\bar{x}) - \psi(x^{*}) - \langle \nabla \psi(x^{*}), \bar{x} - x^{*} \rangle.  \label{eq:94}
	\vspace*{-0.1cm} 
	\end{equation} 
	
	\noindent From  Assumption \ref{ass1}[A.1] and Lemma 2 in \cite{Nes:10}, we have:
	\vspace*{-0.1cm} 
	\begin{equation}
	- \dfrac{L}{2} \|\bar{x} - x^{*}\|^{2} \leq f(\bar{x}) - f(x^{*}) - \langle \nabla f(x^{*}),\bar{x} -x^{*} \rangle,  \label{eq:95}
	\vspace*{-0.1cm} 
	\end{equation}
	
	\noindent where $L = N L_{\max}$ and  $L_{\max} = \max_{i=1:N}  L_{i}$. Using the optimality condition for the problem \eqref{eq:prob}, we have $\nabla f(x^{*}) + \nabla \psi (x^{*})=0$, hence by \eqref{eq:93} and \eqref{eq:95}, we get:
	\vspace*{-0.1cm} 
	\begin{equation*}
	\dfrac{\sigma_q (q-1)}{q} \|\bar{x} - x^{*}\|^{\frac{q}{q-1}} \leq F(\bar{x}) - F(x^{*}) + \dfrac{L}{2} \|\bar{x} - x^{*}\|^{2} .
	\vspace*{-0.1cm} 
	\end{equation*}
	
	\noindent Since $\bar{x} \in \mathcal{L}_{F}(x_{0})$, we have:
	\vspace*{-0.1cm} 
	\begin{equation*}
		F(x_{0}) - F(x^{*}) \geq	F(\bar{x}) - F(x^{*}) \geq \left( \dfrac{\sigma_q (q-1)}{q} \|\bar{x} - x^{*}\|^{\frac{2-q}{q-1}} - \dfrac{L}{2}\right)  \|\bar{x} - x^{*}\|^{2}.
	\end{equation*}
	
	\noindent Combining the inequality above with \eqref{eq:93}, we get:
	\begin{equation*}
		\dfrac{\sigma_q (q-1)}{q} \|\bar{x} - x^{*}\|^{\frac{2-q}{q-1}} - \dfrac{L}{2} > \left( \dfrac{\sigma_q (q-1)}{q} \|\bar{x} - x^{*}\|^{\frac{2-q}{q-1}} - \dfrac{L}{2}\right)  \|\bar{x} - x^{*}\|^{2}, \end{equation*}
		
	\noindent or equivalently	
	\vspace*{-0.1cm} 
	\begin{equation*}
	\dfrac{\sigma_q (q-1)}{q} \|\bar{x} - x^{*}\|^{\frac{2-q}{q-1}} \left( 1 -  \|\bar{x} - x^{*}\|^{2}\right)  \geq \dfrac{L}{2}\left( 1 -  \|\bar{x} - x^{*}\|^{2}\right).  
	\end{equation*}
	
	\noindent From \eqref{eq:114}, we have  $1 -  \|\bar{x} - x^{*}\|^{2} < 0$, hence
	\begin{equation}
	\dfrac{\sigma_q (q-1)}{q} \|\bar{x} - x^{*}\|^{\frac{2-q}{q-1}} \leq \dfrac{L}{2}. \label{eq:115} 
	\end{equation}
Therefore, relation \eqref{eq:115} is a contradiction with \eqref{eq:93}. Hence \eqref{eq:bound} is proved. 
\end{proof}

\section{Convergence analysis: nonconvex case}
Recall that in Lemmas \ref{lem2} and \ref{lem:RCGD2} we proved that the sequence $\left(x_{k}\right)_{k\geq 0}$ generated by the two algorithms CPG or CGD satisfy the following descent for some appropriate positive constant $C$:  

\noindent I) If $i_{k}$ is chosen uniformly at random, then:
\begin{equation}
F(x_{k}) - \mathbb{E}[F(x_{k+1}) \ | \ x_{k}] \geq C\|\nabla F(x_{k})\|^2.  \label{eq:112}
\end{equation}

\noindent II) If $i_{k}$ is chosen  cyclic, then:
\begin{equation}
	F(x_{k}) - F(x_{k+N})  \geq C\|\nabla F(x_{k})\|^2.  \label{eq:132}
\end{equation}

\subsection{Sublinear convergence}
Based on the descent inequalities above, which we proved without requiring the full gradient   $\nabla F$ to be Lipschitz continuous, as it is usually considered in the existing literature, we derive in this section convergence rates for our algorithms depending on  the properties~of~$F$.

\begin{theorem}
	Choose accuracy level $\varepsilon > 0$ and confidence level $\rho \in (0,1)$. Let the sequence  $(x_{k})_{k\geq 0}$ be generated by the algorithms CPG or CGD with $i_{k}$ chosen uniformly at random and satisfying \eqref{eq:112}. If
	\begin{equation}
	k \geq \dfrac{F(x_{0}) - F^{*}}{\varepsilon\rho C}, \label{eq:118}
	\end{equation}
	
	\noindent then in probability we have  
	$$ \mathbb{P}\left[ \displaystyle \min_{0\leq i\leq k-1}  \|\nabla F(x_{i})\|^2 \leq \varepsilon\right] \geq 1 - \rho.$$
\end{theorem}

\begin{proof}
Since 
$ \displaystyle k \min_{0\leq i\leq k-1} \|\nabla F(x_{i})\|^2 \leq  \sum_{i=0}^{k-1} \|\nabla F(x_{i})\|^2$,
we have that
\vspace*{-0.1cm} 
\begin{equation}
\mathbb{P} \left[ \min_{0\leq i\leq k-1} \|\nabla F(x_{i})\|^2 \geq \varepsilon\right] \leq \mathbb{P} \left[ \dfrac{1}{k} \sum_{i=0}^{k-1} \|\nabla F(x_{i})\|^2 \geq \varepsilon \right] . \label{eq:116}
\end{equation}
	
\noindent Further, from  Markov inequality and basic properties of expectation, we get:
\vspace*{-0.1cm} 
\begin{equation}
\mathbb{P} \left[ \sum_{i=0}^{k-1} \|\nabla F(x_{i})\|^2 \geq k \varepsilon \right]  \leq \dfrac{1}{k \varepsilon} \mathbb{E}\left[ \sum_{i=0}^{k-1} \|\nabla F(x_{i})\|^2 \right] \leq \dfrac{1}{k \varepsilon} \sum_{i=0}^{k-1} \mathbb{E}\left[ \|\nabla F(x_{i})\|^2 \right]. \label{eq:117}
\end{equation}
	
\noindent On the other hand, taking the expectation in the inequality \eqref{eq:112}, w.r.t. $\{x_{0},...,x_{k-1}\}$, we have
$
\mathbb{E}[F(x_{k+1})] \leq \mathbb{E}[F(x_{k})] - C \cdot \mathbb{E}[\|\nabla F(x_{k})\|^2]$.
This implies that:
\begin{equation}
	C \sum_{i=0}^{k-1} \mathbb{E}[\|\nabla F(x_{i})\|^2] \leq \sum_{i=0}^{k-1} \left( \mathbb{E}[F(x_{i})] - \mathbb{E}[F(x_{i+1})] \right) \leq  F(x_{0}) - F^{*}, \label{eq:123}
\end{equation}

\noindent Combining the previous relations, we obtain:
\vspace*{-0.1cm} 
\begin{equation*}
\mathbb{P} \left[ \min_{0\leq i\leq k-1} \|\nabla F(x_{i})\|^2 \geq \varepsilon\right] \leq \dfrac{F(x_{0}) - F^{*}}{k \varepsilon C} \leq \rho,
\end{equation*}
which proves our statement.
\end{proof}

\begin{remark}
\label{rem:gradR}
Note that from inequality \eqref{eq:123}, we also obtain: 
\begin{equation*}
 \min_{0\leq i\leq k-1} \mathbb{E} \left[ \|\nabla F(x_{i})\|^2 \right] \leq \dfrac{F(x_{0}) - F^{*}}{kC}.  
\end{equation*}
\end{remark}

\begin{remark}
\label{rem:gradC}
\noindent Using a similar reasoning, in the cyclic case, we also have:
\begin{equation*}
 \min_{0\leq i\leq k-1} 	\| \nabla F(x_{i}) \|^2 \leq \dfrac{N\left( F(x_{0}) -   F^{*}  \right) }{Ck}.
\end{equation*}
\end{remark}



\subsection{Better  convergence under KL}
 In this section we derive convergence rates for our algorithms when the objective function $F$ satisfies the KL property, see Definition  \ref{def:kl}.  In this section we consider the particular form $\kappa (t) = \sigma_q^{\frac{1}{q}} \frac{q}{q-1} t^{\frac{q-1}{q}}$, with $q >1$ and $\sigma_q>0$. Then, the  KL property establishes the following local geometry of the nonconvex function $F$ around a compact set~$\Gamma$:
\vspace*{-0.1cm}
\begin{equation}
	\label{eq:kl}
	F(x) - F_*  \leq \sigma_q \|\nabla F(x)\|^q \quad   \forall x\!: \;  \text{dist}(x, \Gamma) \leq \gamma, \; F_* < F(x) < F_* + \epsilon.  
	\vspace*{-0.1cm}
\end{equation}

\noindent  Note that the relevant aspect of the KL property is when $\Gamma$ is a subset of critical points for $F$, i.e., $\Gamma \subseteq \{x: \nabla F (x)=0 \}$. 
In this section we assume that F satisfies the KL property \eqref{eq:kl} in a subset of critical points of $F$. \blue{In the next lemma, we derive} some basic properties for  $X(x_{0})$,  the limit points of the sequence $(x_{k})_{k\geq 0}$,  and in the proof we use the supermantigale convergence theorem (Theorem 1 in \cite{RobSie:71}). 

\begin{lemma}
	\label{LemmaKL}
	Let the sequence $(x_{k})_{k\geq 0}$  generated by  algorithms CPG or CGD, respectively, be bounded and $i_{k}$ be chosen uniformly  random. If  Assumption \ref{ass1}  and descent \eqref{eq:112} hold, then  $X(x_{0})$ is  compact set,  $F (X(x_{0})) = F_*$,  $F(x_{k})  \to F_{*}$ a.s., and $\nabla F(X(x_{0}))=0$,    $\| \nabla F(x_{k})\|  \to 0$ a.s.
\end{lemma}

\begin{proof}

	Since the  sequence $(x_{k})_{k\geq 0}$ is bounded, this implies that the set $X(x_{0})$ is also bounded. Closedness of $X(x_{0})$ also follows observing that $X(x_{0})$ can be viewed as an intersection of closed sets, i.e., $X(x_{0}) = \cap_{j \geq 0} \cup_{\ell \geq j} \{x_{\ell}\}$. Hence $X(x_{0})$  is a compact set. 	
	Further, using the \blue{boundedness} of $(x_{k})_{k\geq 0}$ and the continuity  of $F$ and $\nabla F$, we have that the sequences $\left( F(x_{k})\right) _{k\geq 0}$ and $\left( \| \nabla F(x_{k})\|^2 \right) _{k\geq 0}$ are also bounded.  Using the  supermantigale convergence theorem  \cite{RobSie:71} and the descent  \eqref{eq:112}, we get (see   \cite{Dav:16}): 
	 \vspace*{-0.2cm}
	\begin{equation}
	 \sum_{k = 0}^{\infty}  \| \nabla F(x_{k})\|^2	 < \infty \quad \text{a.s.}, \;\; \text{hence} \;  \;  \| \nabla F(x_{k})\|  \inas 0.   \label{eq:58}
	  \vspace*{-0.2cm}
	\end{equation}	
\noindent Moreover, we have that $F(x_{k})_{k\geq 0}$ is monotonically decreasing and since $F$ is assumed bounded from below by $F^{*} > - \infty$ (see \eqref{eq:prob}), it converges, let us say to  $F_* > - \infty$, i.e., $F(x_{k}) \inas F_*$ as $k \to \infty$, and   $F_* \geq  F^{*}$.
\noindent Let $x_{*}$ be a limit point of $(x_{k})_{k\geq 0}$, i.e., $x_{*} \in X(x_{0})$. This means that there is a subsequence $(x_{\bar{k}})_{\bar{k}\geq 0}$ of $(x_{k})_{k\geq 0}$ such that $x_{\bar{k}} \inas x_{*}$ as $\bar{k} \to \infty$.  
\noindent	Since $F$ is continuously diferentiable and $x_{\bar{k}} \inas x_{*}$, then we have $F(x_{\bar{k}}) \inas  F(x_{*})$  and $\nabla F(x_{\bar{k}}) \inas \nabla F(x_{*})$. Using basic probability arguments and  \eqref{eq:58}, we get $\nabla F(x_{*}) = 0$  and $F_* = F(x_*)$ a.s.
\end{proof}

\vspace*{-0.1cm}
\begin{remark}
In the deterministic case (i.e., for the cyclic choice of coordinates), using similar arguments as in the previous lemmna we can prove that  the limit points of the sequence $(x_{k})_{k\geq 0}$,  let us say $X(x_{0})$,    is such that $X(x_{0})$ is a compact set, F is constant on $X(x_{0})$  taking value $F_*$ and $\nabla F(X(x_{0}))=0$. 
\label{remarkKLcyclic}		
\end{remark}                   

\noindent In the next theorem, based on the results of the previous lemma,  we assume that $F$ satisfies the KL condition \eqref{eq:kl} with constant value $F_{*}$ and constant $q \in(1,2]$  around the limit points of the sequence $(x_{k})_{k\geq 0}$,  denoted $X(x_{0})$.  From previous lemma we have that $F(x_{k}) \inas F_{*}$, which means that there exists some measurable set $\Omega$ such that $ \mathbb{P}[\Omega] =1$ and for any  $\epsilon,\gamma>0$ and $\omega \in \Omega$  there exists
\red{$k_{\epsilon,\gamma}(\omega)$} such that for any \red{$k \geq k_{\epsilon,\gamma}(\omega)$} we have   $F(x_k(\omega)) - F_{*} \leq \sigma_q \| \nabla F(x_k(\omega))\|^q$. Note that we cannot infer from this that $F(x_k) - F_* \leq \sigma_q \| \nabla F(x_k)\|^q$  for $k$ large enough as $k_{\epsilon,\gamma}(\omega)$ is a random variable which, in general,  cannot be  bounded uniformly on $\Omega$.  However, using similar  arguments as in  Theorem 4.5 in \cite{MauFadAtt:22}, which invokes measure theoretic arguments to pass from almost sure convergence to almost uniform convergence, thanks to
Egorov’s theorem \red{(see Theorem 4.4 in \cite{SteSha:05})}, we can prove that  for any \red{ $\delta,  \epsilon ,\gamma >0$} there exist a measurable set $\Omega_{ \delta} \subset \Omega$, such that $ \mathbb{P}[\Omega_{\delta}] \geq 1 - \delta$,  and \red{$k_{\delta,\epsilon,\gamma} > 0$} such that for all $\omega \in \Omega_{ \delta} $ and \red{$k \geq k_{ \delta,\epsilon,\gamma} $} we have $F(x_k(\omega)) - F_* \leq \sigma_q \| \nabla F(x_k(\omega))\|^q$. 
Hence, with probability at least $1-\delta$ the sequence $(x_{k})_{k\geq0}$ satisfies KL on $\Omega_{\delta}$ for \red{ $k \geq k_{\delta,\epsilon,\gamma}$. For simplicity, define  $C_{0} =  F(x_{0}) - F_{*} $ and $\mathbbm{1}_{A}$ the  indicator function of a set $A$} and \blue{ recall that $\gamma$ and $\epsilon$ are constants from the KL inequality \eqref{eq:kl}.  }
\vspace*{-0.1cm}
\begin{theorem}
	\label{theo:KL1}
		Let $X(x_{0})$ be the set of limit points of the sequence $(x_{k})_{k\geq 0}$ generated by  algorithms CPG or CGD, with $i_{k}$ chosen uniformly at random. If  the descent \eqref{eq:112} holds and $F$ satisfies the KL property \eqref{eq:kl} on $X(x_{0})$, with $q \in(1,2]$ and constant value $F_{*}$, then for any $ \delta > 0$ there exist a measurable set $\Omega_{ \delta}$ satisfying  $ \mathbb{P}[\Omega_{ \delta}] \geq 1 - \delta$ and \red{ $k_{ \delta,\gamma,\epsilon} > 0$} such that with probability at least $1 - \delta$ the following statements hold for all \red{$k\geq k_{\delta,\epsilon,\gamma}$}:  
	
	\noindent  (i) If $q\in (1,2)$, we have the following sublinear rate:
	 \vspace*{-0.2cm}
\red{	\begin{align}
		\mathbb{E}[F(x_{k}) - F_*] \leq \dfrac{ q^{\frac{q}{2-q}} C^{\frac{2-q}{q}}  \sigma_{q}^{-\frac{2(2-q)}{q^2}}  }{\left( (k-k_{ \delta,\epsilon,\gamma})(2-q) \right)^\frac{q}{2-q}} + C_{0}\sqrt{\delta}.  \label{eq:137} 
		 \vspace*{-0.2cm}
	\end{align}}

	\noindent  (ii) If $q=2$,  we have the following linear rate: 
\red{	\begin{align}
	\mathbb{E}[F(x_{k}) - F_{*}] 
	\leq (1- C \sigma_{2}^{-1})^{k -  k_{ \delta,\epsilon,\gamma}}  \mathbb{E}[F(x_{k_{ \delta,\epsilon,\gamma}}) - F_{*}]    + C_{0}\sqrt{\delta}. 
		\label{eq:51}
	\end{align}}

	
\end{theorem} 

\begin{proof}
	From Lemma \ref{LemmaKL}, we have that $F(x_{k}) \inas F_*$ and $\|\nabla F(x_{k}) \| \inas 0$,  i.e., there exists a set $\Omega$ such that  $ \mathbb{P}[\Omega]  = 1$ and for all $\omega \in \Omega: F(x_{k}(\omega)) \to F_*(\omega)$ and $\|\nabla F(x_{k}(\omega)) \| \to 0$.	
	 Moreover, from the Egorov’s theorem \red{(see Theorem 4.4 in \cite{SteSha:05}), we have that for any  $ \delta>0$ there exists a measurable set $\Omega_{ \delta} \subset \Omega$ satisfying $ \mathbb{P}[\Omega_{ \delta}] \geq 1 - \delta$ such that  $F(x_{k})$ converges uniformly to $F_{*}$ and $\nabla F(x_{k})$  converges uniformly to $0$  on the set $\Omega_{ \delta}$.
Since $F$ satisfies the KL property,  given $\epsilon, \gamma, \delta > 0$, there exists a $k_{ \delta,\epsilon,\gamma} > 0$ and $\Omega_{\delta} \subset \Omega$ with $P[\Omega_{\delta}] \geq 1 - \delta$ such that  $\text{dist}(x_k(\omega), X(x_{0})) \leq \gamma, \; F_* < F(x_{k}(\omega)) < F_* + \epsilon$ for all $k \geq k_{\delta,\epsilon,\gamma}$ and $\omega \in \Omega_{ \delta}$ and additionally: 
	\begin{equation}
		\label{eq:139}
		F(x_{k}(\omega)) - F_*  \leq \sigma_q \|\nabla F(x_{k}(\omega))\|^q \quad \forall k \geq k_{\delta,\epsilon,\gamma} \text{ and } \omega \in \Omega_{ \delta}.
	\end{equation}}
	\noindent\red{ Equivalently, we have:
	\begin{equation*}
	  \mathbbm{1}_{\Omega_{ \delta}}	\left(  F(x_{k}) - F_* \right) \leq  \mathbbm{1}_{\Omega_{ \delta}} \sigma_q \|\nabla F(x_{k})\|^q \quad \forall k \geq k_{\delta,\epsilon,\gamma}. 
	\end{equation*}}
	\noindent \red{Taking expectation on both sides of the previous inequality and since $F(x_{k}) \leq F(x_{0})$, from  Lemma \ref{lem:prob} in Appendix we have for all $k \geq k_{\delta,\epsilon,\gamma}$:
	\begin{align*}
		 \mathbb{E}[ F(x_{k}) - F_*] - C_{0} \sqrt{\delta} &\leq 
		\mathbb{E}[\mathbbm{1}_{\Omega_{ \delta}}	\left( F(x_{k}) - F_* \right) ]  \leq  \mathbb{E}[\mathbbm{1}_{\Omega_{ \delta}} \sigma_q \|\nabla F(x_{k})\|^q ] \\
		&\leq \mathbb{E}[ \sigma_q \|\nabla F(x_{k})\|^q ] = \sigma_{q} \mathbb{E}[  \|\nabla F(x_{k})\|^q ].  
	\end{align*}}
	 \noindent Since for $q \in (1,2]$,  $t \mapsto t^{\frac{2}{q}}$ is a convex function on $\mathbb{R}_+$, then we obtain: 
	\red{ \begin{align*}
	 	\left(  \mathbb{E}[ F(x_{k}) - F_*] - C_{0} \sqrt{\delta}\right)^{\frac{2}{q}} \leq  \sigma_{q}^{\frac{2}{q}} \left( \mathbb{E}[  \|\nabla F(x_{k})\|^q ]\right) ^{\frac{2}{q}} \leq \sigma_{q}^{\frac{2}{q}} \mathbb{E}[  \|\nabla F(x_{k})\|^2 ].
	 \end{align*}}
	\noindent Taking also expectation on both sides of the inequality \eqref{eq:112} and combining with the inequality above,  we get:
	\begin{align}
		C \sigma_{q}^{-\frac{2}{q}} \left(  \mathbb{E}[ F(x_{k}) - F_*] - C_{0} \sqrt{\delta}\right)^{\frac{2}{q}} \leq  \mathbb{E}[F(x_{k}) - F_*]  - \mathbb{E}[F(x_{k+1}) - F_*]. \label{eq:64}
	\end{align}
	
	\noindent\red{ Note that, if   $ \mathbb{E}[F(x_{\bar{k}}) - F_*] \leq C_{0}\sqrt{\delta}$ for some $\bar{k} \geq  k_{\delta,\epsilon,\gamma}$ then \eqref{eq:137} and \eqref{eq:51} are satisfied for $k \geq \bar{k}$, since $(F(x_{k}))_{k\geq 0}$ is decreasing. Otherwise, if $ \mathbb{E}[F(x_{k}) - F_*] > C_{0}\sqrt{\delta}$}, first,  consider $q\in(1,2)$ and define $\gamma_c = C \sigma_{q}^{-\frac{2}{q}}$.
	Multiplying both sides of \eqref{eq:64} by $\gamma_c^{\frac{q}{2-q}}$,  we obtain:  
	\begin{equation}
		\label{eq:138}
		\theta_{k} - \theta_{k+1} \geq  (\theta_{k})^{\frac{2}{q}} \quad \text{for} \quad k \geq k_{ \delta,\epsilon, \gamma}, 	 
	\end{equation}
\red{	with $$\theta_{k} = \gamma_c^{\frac{q}{2-q}}\left(  \mathbb{E}[F(x_{k}) - F_*] - C_{0}\sqrt{\delta}\right). $$ }
	 Considering the \red{ second inequality of Lemma 9 in \cite{NecCho:21}} for $\zeta = \frac{2-q}{q}>0$,  we get:
	\vspace*{-0.2cm} 
	\begin{align*}
		& \theta_{k} \leq \dfrac{1}{\left( \zeta (k - k_{\delta,\epsilon,\gamma})\right)^\frac{1}{\zeta}} \iff \\ 
		&  \gamma_c^{\frac{q}{2-q}}\left( \mathbb{E}[F(x_{k}) - F_*] - C_{0}\sqrt{\delta}\right)  \leq \dfrac{ 1}{\left( \frac{(k-k_{ \delta,\epsilon,\gamma})(2-q)}{q}\right)^\frac{q}{2-q}}  \iff \\
		&  \mathbb{E}[F(x_{k}) - F_*] - C_{0}\sqrt{\delta} \leq \dfrac{ q^{\frac{q}{2-q}}\gamma_c^{\frac{2-q}{q}}}{\left( (k-k_{ \delta,\epsilon,\gamma})(2-q) \right)^\frac{q}{2-q}},
		\vspace*{-0.2cm} 
	\end{align*}
	
	\noindent proving the first statement of the theorem. Second, if $q=2$,  by Lemma 9 in \cite{NecCho:21} and \eqref{eq:64}, we have: 	
	$
	\mathbb{E}[F(x_{k})- F_{*}] - C_{0}\sqrt{\delta}  \leq \left(1-\dfrac{C}{\sigma_{2}}\right)^{(k-k_{ \delta,\epsilon,\gamma})} ( \mathbb{E}[F(x_{k_{ \delta,\epsilon,\gamma}}) - F_{*}] - C_{0}\sqrt{\delta})$, and then  \eqref{eq:51} follows.	
\end{proof}

The next lemma is an extension of  a result   in \cite{RicTak:11}. \red{ Note that in \cite{RicTak:11}, the case  $\zeta =1$ was considered  and in the next lemma we derive the result for any  $\zeta>0$.} For completeness, we give its proof in Appendix.    

\begin{lemma}
	\label{lem:recprob}
	Fix $x_0 \in \mathbb{R}^{n}$ and let $(x_{k})_{k\geq 0}$ be a sequence of random vectors in $\mathbb{R}^{n}$ with $x_{k+1}$ depending only on  $x_{k}$. Let $\phi : \mathbb{R}^{n} \to \mathbb{R}$ be a nonnegative function and define $\Delta_{k} = \phi \left( x_{k} \right)$.  Lastly, let $\zeta > 0$ , choose accuracy level  $0 < \varepsilon < \Delta_{0}$, with $\varepsilon \in (0,1)$, confidence level $\rho \in (0,1)$, and assume that the sequence of random variables $(\Delta_{k})_{k\geq 0}$ is nonincreasing and has the following property:
\red{	\begin{equation}
	\mathbb{E} [\Delta_{k+1}] \leq \mathbb{E} [\Delta_{k}] - \mathbb{E} [\Delta_{k}]^{\zeta + 1} \quad \forall k \geq \bar{k}. \label{eq:84}
	\end{equation}}
	
	\noindent If
	\vspace*{-0.3cm}
	\begin{equation}
	k \geq \dfrac{1}{\zeta}\left(\dfrac{1}{\varepsilon^{\zeta}}  - \dfrac{1}{ \Delta_{0}^{\zeta}}\right)  + 2 + \dfrac{1}{\varepsilon^{\zeta}}\log\dfrac{1}{\rho} + \bar{k},  \;\; \text{then}  \quad  
	 \mathbb{P}\left[\Delta_{k} \leq \varepsilon \right] \geq 1 - \rho. \label{eq:89}
	\end{equation}
\end{lemma}

\noindent \red{ Next, combining previous lemma with Theorem \ref{theo:KL1}, we can also derive convergence results in probability, when the function $F$ satisfies the KL condition \eqref{eq:139}.}

\begin{theorem}
		Let $X(x_{0})$ be the set of limit points of the sequence $(x_{k})_{k\geq 0}$ generated by the algorithm CPG or CGD, respectively, with $i_{k}$ chosen uniformly  random. Assume that  the descent  \eqref{eq:112} holds and  $F$ satisfies the KL property \eqref{eq:kl} on $X(x_{0})$, with $q\in(1,2]$ and constant value $F_{*}$. Further, choose accuracy level $\varepsilon \in (0,1)$ and confidence level $\rho \in (0,1)$.  Then, for any $ \delta > 0$  there exist  $k_{ \delta,\epsilon,\gamma} > 0$ such that with probability at least  $1- \delta$ we have:
		\vspace*{-0.1cm} 
\red{	\begin{equation*}
 \text{if} \;\;  q\in(1,2) \;\;  \text{and} \;\;   	k \geq \dfrac{q}{2-q} \left( \dfrac{1}{\varepsilon^{\frac{2-q}{q}}} - \frac{\sigma_{q}^{\frac{2}{q}}}{  C \left(F(x_{0}) - F_{*}\right)^{\frac{2-q}{q}}}\right) + 2 + \dfrac{1}{\varepsilon^{\frac{2-q}{q}}}\log\dfrac{1}{\rho} + k_{ \delta,\epsilon,\gamma},
	\vspace*{-0.1cm} 
	\end{equation*}
	\noindent or 
	\vspace*{-0.1cm}
	\begin{equation}
 \text{if} \;\; q=2  \;\;  \text{and} \;\;   k \geq \dfrac{\sigma_{2}}{C} \log\left( \dfrac{F(x_{0}) - F_{*} }{\varepsilon\rho}\right) + k_{ \delta,\epsilon,\gamma}, \label{eq:98}
	\end{equation}
	\noindent then $$ \mathbb{P}[F(x_{k}) - F_{*} \leq \varepsilon + C_{0} \sqrt{\delta}] \geq 1 - \rho. $$}
	\vspace*{-0.2cm}
\end{theorem}

\begin{proof}
\red{If $ \mathbb{E}[F(x_{\bar{k}}) - F_*] \leq C_{0}\sqrt{\delta}$, \blue{Markov's} inequality directly implies the result. On other hand, if  $ \mathbb{E}[F(x_{{k}}) - F_*] \geq C_{0}\sqrt{\delta}$,  $q\in(1,2)$  and  $F$ satisfies the KL property \eqref{eq:139}, using \eqref{eq:138} and \eqref{eq:84} with $\zeta = \frac{2-q}{q}>0$, we get the result.	
For $q=2$,	\blue{from Markov's} inequality and \eqref{eq:51}, we have for all $k\geq k_{ \delta,\epsilon,\gamma}$: 
	\begin{align*}
	\mathbb{P}[F(x_{k}) - F_{*} - C_{0} \sqrt{\delta} \geq \varepsilon] &\leq \dfrac{1}{\varepsilon} \left( \mathbb{E} [F(x_{k}) - F_{*}] - C_{0} \sqrt{\delta}\right) \\
	& \leq \dfrac{1}{\varepsilon} (1-C \sigma_{2}^{-1})^{k-k_{ \delta,\epsilon,\gamma}} \left( F(x_{0}) - F_{*}\right). 
	\vspace*{-0.1cm} 
	\end{align*}
	\noindent Using \eqref{eq:98}, we obtain 
	$ \mathbb{P}[F(x_{k}) - F_{*}\geq \varepsilon + C_{0}\sqrt{\delta}] \leq \rho. $ }
\end{proof}

\noindent \red{ Now we are ready to present the convergence results in the cyclic case when the function $F$ satisfies the KL condition \eqref{eq:kl} with constant value $F_{*}$ and constant $q >1$  around the limit points of the sequence $(x_{k})_{k\geq 0}$,  denoted $X(x_{0})$.  Note that, in this case we can also have a superlinear rate when $q>2$.}
\begin{theorem}
	\label{the:KLC}
	Let $(x_{k})_{k\geq 0}$  be the sequence generated by  algorithm CPG or CGD, respectively, with $i_{k}$ chosen cyclic . If  the descent \eqref{eq:132} holds and $F$ satisfies the KL property \eqref{eq:kl} on $X(x_{0})$, with \red{$q >1$} and constant value $F_{*}$, then we have the following convergence rates:

	\noindent  (i) If $q\in (1,2)$, there exists $\bar{k}_{\epsilon,\gamma} > 0$ such that the following sublinear rate holds:
	\vspace*{-0.1cm} 	
	\begin{eqnarray*}
		F(x_{kN}) - F_{*} \leq \dfrac{F(x_{\bar{k}_{\epsilon,\gamma}N }) - F_{*}}{\left(\frac{(2-q)}{q} C \sigma_{q}^{-\frac{2}{q}} (F(x_{\bar{k}_{\epsilon,\gamma}N })-F_{*})^{\frac{2-q}{q}} \cdot (k - \bar{k}_{\epsilon,\gamma}  ) + 1 \right)^{\frac{q}{2-q}}}  \quad \forall  k \geq \bar{k}_{\epsilon,\gamma}. 
		\vspace*{-0.1cm} 
	\end{eqnarray*}

	\noindent  (ii) If $q=2$, there exists $\bar{k}_{\epsilon,\gamma} > 0$ such that the following linear rate holds:
	\vspace*{-0.1cm} 	
	\begin{eqnarray*}
		F(x_{kN}) - F_{*} \leq (1- C \sigma_{2}^{-1})^{(k-  \bar{k}_{\epsilon,\gamma} )} \left( F(x_{\bar{k}_{\epsilon,\gamma}N }) - F_{*} \right)  \quad \forall k \geq \bar{k}_{\epsilon,\gamma}. 
		\vspace*{-0.1cm} 
	\end{eqnarray*}
	\noindent \red{ (iii) If $q>2$ we have the following superlinear rate: 
	\begin{equation*}
	F(x_{kN}) - F_{*} \leq \left(\dfrac{1}{1 + C \sigma_{q}^{-\frac{2}{q}} \left( F(x_{kN}) - F_{*} \right)^{\frac{2}{q}-1}} \right) ( F(x_{(k-1)N}) - F_{*})  \quad \forall k > \bar{k}_{\epsilon,\gamma}.
	\label{eq:53}
	\end{equation*}}
	
\end{theorem}

\begin{proof}
\red{From Remark \ref{remarkKLcyclic}, we have that there exists a $\bar{k}_{\epsilon,\gamma}>0$ such that the KL property \eqref{eq:kl} holds for all $k \geq \bar{k}_{\epsilon,\gamma}$. Combining the KL property \eqref{eq:kl} with the descent inequality \eqref{eq:132},  we obtain for all $k \geq \bar{k}_{\epsilon,\gamma}$: 
\begin{equation}
\left(F(x_{k}) - F_*\right)^{\frac{2}{q}}  \leq \sigma_{q}^{\frac{2}{q}} \|\nabla F(x_{k})\|^{2} \leq \sigma_{q}^{\frac{2}{q}} C^{-1} \left( F(x_{k}) - F(x_{k+N}) \right).  
\end{equation}
\noindent Considering $k=\hat{k}N$ in the inequality above, with $\hat{k} \geq \frac{\bar{k}_{\epsilon,\gamma}}{N} $,  we get: 
\begin{equation}
\left( F(x_{\hat{k}N}) - F_*\right)  - \left( F(x_{(\hat{k}+1)N}) - F_*\right)  \geq  C \sigma_{q}^{-\frac{2}{q}}\left(  F(x_{\hat{k}N}) - F_*\right) ^{\frac{2}{q}}.  
\end{equation}
\noindent Define $\Delta_{\hat{k}} = F(x_{\hat{k}N}) - F_*$. Using Lemma 9 in \cite{NecCho:21} and similar arguments as in Theorem \ref{theo:KL1}, we get the statements.}
\end{proof}


\section{Convergence analysis: convex case}
In this section we assume  that  the composite objective function  $F=  f + \psi$ is convex. Note that we do not need to impose convexity on $f$ and $\psi$ separately. Denote the set of optimal solutions of \eqref{eq:prob} by $X^{*}$ and let  $x^{*}$ be an element of this set. Define also: 
	\vspace*{-0.2cm}
\begin{equation*}
R = \max_{k \geq 0} \min_{x^{*} \in X^{*}} \|x_{k} - x^{*}\| < \infty. 
\vspace*{-0.1cm} 
\end{equation*}

\begin{theorem}
	\label{the:conv}
	Let $(x_{k})_{k\geq 0}$ be generated by algorithm CPG or CGD,  with $i_{k}$ chosen uniformly at random. If  the descent \eqref{eq:112} holds and  $F$ is  convex, then the following sublinear rate in function values holds:  
	\begin{equation}
	\mathbb{E}[F(x_{k}) - F(x^{*})] \leq \dfrac{\left(F(x_{0}) - F(x^{*})\right) R^2 }{C\left(F(x_{0}) - F(x^{*})\right)k + R^2}.  \label{eq:17}
	\end{equation}
\end{theorem}

\begin{proof}
	Since $F$ is convex, we have:
	\begin{equation*}
		F(x^{*}) - F(x_{k}) \geq \langle \nabla F(x_{k}), x^{*} - x_{k+1} \rangle \geq - \| \nabla F(x_{k})\| \|x_{k+1} - x^{*} \| \geq - \| \nabla F(x_{k})\| R.
	\end{equation*}
	\noindent Hence,
		\vspace*{-0.2cm}
	\begin{equation}
	 \|\nabla F(x_{k})\| \geq \dfrac{F(x_{k}) - F(x^{*})}{R}.  \label{eq:133}
	\end{equation}
	Combining this inequality with \eqref{eq:112}, we get: 
	\vspace*{-0.1cm}
	\begin{equation}
	\left(F(x_{k}) - F(x^{*})\right)  - \mathbb{E}[F(x_{k+1}) - F(x^{*})  \ | \ x_{k} ]  \geq  C \dfrac{\left( F(x_{k}) - F(x^{*})\right)^{2}}{R^{2}}. \label{eq:99}
	\end{equation} 
	
	\noindent Since $t \mapsto t^{2}$ is convex function, then taking expectation on both sides of the inequality \eqref{eq:99} w.r.t. $\{x_{0},...,x_{k-1}\}$, we obtain:
	\vspace*{-0.1cm}
	\begin{equation*}
	\mathbb{E}[F(x_{k}) - F(x^{*})]  -  \mathbb{E}[F(x_{k+1}) - F(x^{*})]  \geq  C \dfrac{\mathbb{E}[F(x_{k}) - F(x^{*})]^{2}}{R^{2}}.
	\end{equation*}   
	
	\noindent Multiplying both sides by $C/R^{2}$, we further get: 
	\begin{equation*}
	\dfrac{C \cdot \mathbb{E}[F(x_{k}) - F(x^{*})]}{R^{2}}  - \dfrac{C \cdot \mathbb{E}[F(x_{k+1}) - F(x^{*})]}{R^{2}}  \geq  \left[\dfrac{C \cdot \mathbb{E}[F(x_{k}) - F(x^{*})]}{R^{2}}\right]^{2}.
	\end{equation*} 
	
	\noindent Denote
	$\Delta_{k} = \dfrac{C \cdot \mathbb{E}[F(x_{k}) - F(x^{*})]}{R^{2}}$. Therefore, we obtain the following recurrence: 
	$
	\Delta_{k} - \Delta_{k+1} \geq \left(\Delta_{k}\right)^{2}.
	$
	From Lemma 9 in \cite{NecCho:21},  we obtain:
	\vspace*{-0.1cm} 
	\begin{equation*}
	\dfrac{C \cdot \mathbb{E}[F(x_{k}) - F(x^{*})]}{R^{2}} \leq \dfrac{C\left(F(x_{0}) - F(x^{*})\right)}{C\left(F(x_{0}) - F(x^{*})\right) k + R^2} . 
	\end{equation*}
	\noindent which proves our statement.
\end{proof}

\begin{theorem}
	 Choose accuracy level $\varepsilon \in (0,1)$ and confidence level $\rho \in (0,1)$. Let $(x_{k})_{k\geq 0}$ be generated by the algorithms CPG or CGD, with $i_{k}$ chosen uniformly at random,  and assume that the descent \eqref{eq:112} holds. If $F$ is convex function and
	\begin{equation*}
	k \geq \dfrac{1}{\varepsilon}\left( 1+\log \dfrac{1}{\rho}\right)  + 2 - \dfrac{R^2}{C(F(x_{0}) - F^{*})},
	\end{equation*}
	
	\noindent then
	$$\mathbb{P}[F(x_{k}) - F^{*} \leq \varepsilon] \geq 1 - \rho.$$
	
\end{theorem}

\begin{proof}
	Multiplying both sides of \eqref{eq:99} by $C/R^{2}$, we obtain:
	$$ \mathbb{E}[\Delta_{k+1} \ | \ x_{k} ]  \leq \Delta_{k}  -  \Delta_{k}^{2},$$
	 with $ \Delta_{k} =	\dfrac{C \left( F(x_{k}) - F(x^{*})\right) }{R^{2}}$. Using Theorem 1 from \cite{RicTak:11}, the statement follows.	
\end{proof}

\begin{theorem}
	\label{the:conC}
	Let $(x_{k})_{k\geq 0}$ be generated by algorithm CPG or CGD,  respectively, with $i_{k}$ chosen cyclic. If  the descent \eqref{eq:132} holds and  $F$ is  convex, then the following sublinear rate in function values holds:  
	\begin{equation*}
		F(x_{kN}) - F(x^{*}) \leq \dfrac{\left(F(x_{0}) - F(x^{*})\right) R^2 }{C\left(F(x_{0}) - F(x^{*})\right)k + R^2}. 
	\end{equation*}
\end{theorem}

\begin{proof}
\red{From inequalities \eqref{eq:132} and \eqref{eq:133},  we obtain for all $k\geq 0$:
	\vspace*{-0.2cm}
\begin{equation}
	\left( F(x_{k}) - F(x^*) \right)  - \left( F(x_{k+N}) - F(x^*) \right)   \geq C \dfrac{\left( F(x_{k}) - F(x^{*})\right)^{2}}{R^{2}}.  
		\vspace*{-0.2cm}
\end{equation}	
\noindent Considering $k=\hat{k}N$ in the inequality above, with $\hat{k} \geq 0$, we get: 	
	\vspace*{-0.2cm}
\begin{equation}
	\left( F(x_{\hat{k}N}) - F(x^*) \right)  - \left( F(x_{(\hat{k}+1)N}) - F(x^*) \right)   \geq C \dfrac{\left( F(x_{\hat{k}N}) - F(x^{*})\right)^{2}}{R^{2}}.  
		\vspace*{-0.2cm}
\end{equation}	
\noindent Define $\Delta_{\hat{k}} = F(x_{\hat{k}N}) - F_*$. Using Lemma 9 in \cite{NecCho:21} and similar arguments as in Theorem  \ref{the:conv}, we get the statement. }
\end{proof}


\section{Numerical simulations}
\label{Simulations}
In the numerical  experiments we consider two applications:  the subproblem in the cubic Newton method \cite{NesPol:06} and the orthogonal matrix factorization problem \cite{AhoHie:21}. In the sequel, we describe these problems, provide some implementation details and present the numerical results. Note  that our composite problem \eqref{eq:prob} permits to handle general coupling functions $\psi(x)$,   with $x= (x_1,\cdots ,x_N)$, e.g.:  (i) $\psi(x) = \| A x\|^p$, with $p \geq 2$ and $A$ linear operator (in particular, $\psi(x_1,x_2) = \|A_1x_1 - A_2 x_2\|^p$, see \cite{LuFreNes:18}); (ii) when solving the  subproblem in higher order methods (including cubic Newton) recently popularized by Nesterov \cite{Nes:19} (where $\psi(x) = \|x\|^p$); (iii)  when  minimizing an objective function that is relatively smooth w.r.t. some (possibly unknown) function $h$, see \cite{LuFreNes:18}.

\subsection{Cubic Newton}
In the first set of experiments, we consider solving the subproblem in the cubic Newton method, an algorithm which  is supported by global efficiency estimates for general classes of optimization problems \cite{NesPol:06}. In each iteration of the cubic Newton one needs to  minimize an objective function of the form:
\vspace*{-0.2cm} 
\begin{equation}
\label{opt_prob}
\min_{x \in \mathbb{R}^{n}} F(x)   :=  \frac{1}{2}   \langle Ax, x \rangle  + \langle b, x \rangle + \frac{M}{6} \|x \|^3,
\vspace*{-0.1cm} 
\end{equation}
where $A \in \mathbb{R}^{n\times n}, b\in\mathbb{R}^{n}$ and $M>0$ are given. Note that the function   $\psi(x) = \frac{M}{6} \|x \|^3$ is  uniformly convex with $\sigma_3 = \frac{M}{4}$, but it is nonseparable and twice differentiable. Moreover, in this case $f(x) =  \langle b, x \rangle + \frac{1}{2}   \langle Ax, x \rangle$ is smooth. Hence, this problem fits into our general model \eqref{eq:prob} and we can use algorithm CPG to solve it. Moreover, this problem satisfies both conditions A.4 and A.5 from Assumption \ref{ass2}. Therefore, we can also use the algorithm  CGD  with the first stepsize choice (i.e., equation \eqref{eq:HF1}) for solving the problem \eqref{opt_prob}.    
In the simulations we use the  stopping criteria:
$\|\nabla F(x_{k})\| \leq 10^{-2}$.
In Table \ref{table:1}, ``**" means that the corresponding algorithm needs more than 5 hours to solve the problem.    For the symmetric matrix $A$, we consider the eigenvalues of $A$ ordered  as $\lambda_{1}(A) \geq \cdots \geq \lambda_{n}(A)$. 

\medskip 

\noindent \textbf{Implementation details for CPG algorithm:}
Note that  at each iteration  of CPG for solving problem \eqref{opt_prob} we need to solve a subproblem of the  form:
\begin{align*}
d_{k} = \arg \min_{d \in \mathbb{R}^{n_{i_{k}}}} \langle U_{i_{k}}^{T}\nabla f(x_{k}), d \rangle + \frac{\Hf}{2} \|d\|^2 + \frac{M}{6} \|x_{k} + U_{i_{k}}d\|^3.
\end{align*}

\noindent As proved in \cite{NecCho:21}, solving the previous subproblem is equivalent to finding a positive root of  the following  fourth order equation: 
\vspace*{-0.1cm} 
\begin{eqnarray}
&& \dfrac{M^{2}}{4} \mu^{4} + \Hf M \mu^{3} + \left(\Hf^{2} -  \dfrac{M^{2}}{4} \sum_{j \neq i_{k}} \|x_{k}^{(j)}\|^{2} \right)  \mu^{2} - \Hf M \sum_{j \neq i_{k}} \|x_{k}^{(j)}\|^{2} \mu  \nonumber \\
 && \qquad \qquad - \|\Hf x_{k}^{(i_{k})}  - (U_{i_{k}}U_{i_{k}}^{T}\nabla f(x))_{i_{k}}\|^{2} - \Hf^2 \sum_{j \neq i_{k}}\|x_{k}^{(j)}\|^{2} = 0, \label{fourth:1}
\vspace*{-0.1cm} 
\end{eqnarray}
\noindent where $x_{k} = (x_{k}^{(1)},\cdots,x_{k}^{(n)})^{T}$. Once we compute a positive root $\mu_k$, then for the update we use: 
$d_{k} = -\dfrac{2U_{i_{k}}^{T}\nabla f(x_{k})  + \mu_k M U_{i_{k}}^{T}x_{k}}{2\HfU + \mu M}$. Note that, the fourth order equation  \eqref{fourth:1}  has only one change of sign. Then, using Descarte's rule of signs \cite{Mes:82} we have that the equation \eqref{fourth:1} has only one positive root. 

\medskip 

\noindent \textbf{Implementation details for CGD algorithm:}
Note that the hessian $\nabla^2 \psi (x) = M xx^{T}/ (2\|x\|) + M/2 \|x\|I_{n}$ satisfies the following inequality $\|U_{i_{k}}^{T}\nabla^2 \psi (x)U_{i_{k}}\| \!\leq\! M \|x\|$. Thus,  condition [A.4] in Assumption \ref{ass2} holds with $p=1$ and $\Hp = M$. Moreover,    [A.5] in Assumption \ref{ass2} holds  with $L_{\psi} = M$, see \cite{Nes:08}. Therefore, we can apply the CGD method, with $\HF$ given by the first stepsize choice (i.e., equation \eqref{eq:HF1}), for solving the problem \eqref{opt_prob}. Note that according to the first stepsize choice we need  at  each iteration to  compute a positive root $\alpha_k$  of the following second order equation:  
\vspace*{-0.2cm}
\begin{equation*}
\dfrac{M}{6} \alpha^{2} + \left(  \dfrac{M}{2} \|x_{k}\| + \Hf\right) \alpha - \|U_{i_{k}}^{T} \nabla F(x_{k})\| = 0, 
	\vspace*{-0.2cm}
\end{equation*}
\noindent and then 
$\HF = \frac{M}{2}\|x_{k}\| + \frac{M}{6} \alpha_{k} + \Hf$.   We implemented the following algorithms: 

\medskip 

\noindent    1) RCPG:  CPG with random $i_{k}$  $N=n$ and $\Hf = |U^{T}_{i_k}AU_{i_k}|$.  \\
	2)  RCGD-1: CGD with random $i_{k}$, $N=n$  and $\Hf = 0.51 \cdot |U^{T}_{i_k}AU_{i_k}|$. \\
	 3) RCGD-2: CGD with random $i_{k}$, $N=n$ and $\Hf =  |U^{T}_{i_k}AU_{i_k}|$. \\
	  4) CCPG: CPG with cyclic $i_{k}$, $N=n$  and $\Hf = |U^{T}_{i_k}AU_{i_k}|$.  \\
	  5) CCGD-1: CGD with cyclic $i_{k}$, $N=n$  and $\Hf = 0.51 \cdot |U^{T}_{i_k}AU_{i_k}|$. \\
	  6) CCGD-2: CGD with cyclic $i_{k}$, $N=n$ and $\Hf =  |U^{T}_{i_k}AU_{i_k}|$. \\
	  7) GD-1: CGD algorithm with  $N=1$ and $\Hf = 0.51 \cdot |A|$. \\
	  8) GD-2: CGD with  $N=1$  and $\Hf =  |A|$. \\
	  9) Algorithm (46) in \cite{Nes:19} and gradient method proposed in \cite{CarDuc:17}.  The only difference between the method in \cite{CarDuc:17} and our variants GD-1 and GD-2 consists in  how the stepsize is defined. In the method proposed in  \cite{CarDuc:17} the stepsize is constant, while  in  our GD-1 and GD-2  the stepsizes are adaptive. \\
	10)  GD (line-search):  gradient method with Armijo line search from \cite{Arm:66}.\\
 	11)  RCGD (line-search): Coordinate gradient method with Armijo line-search and {random} $i_{k}$, $N=n$ ( Algorithm 2.1 in \cite{Bon:21} with  $\beta = \delta_i = \dfrac{1}{2}$) .\\
 	12)  CCGD (line-search): Coordinate gradient method with Armijo line search and cyclic $i_{k}$, $N=n$ (variant of Algorithm 2.1 in \cite{Bon:21}).

\medskip 

\noindent In the first set of experiments, the vector $b\in\mathbb{R}^{n}$ was generated from a standard normal distribution $\mathcal{N}(0,1)$ and the matrix $A \in \mathbb{R}^{n \times n}$ was generated as $A = Q^{T}BQ$, where $Q \in \mathbb{R}^{n \times n}$ is an orthogonal matrix  and $B \in \mathbb{R}^{n \times n}$  is a diagonal matrix with real  entries. Following  \cite{CarDuc:17}, the starting point is chosen as: 
\vspace*{-0.3cm}
\begin{equation*}
x_{0} = -r\dfrac{b}{\|b\|},  \text{ with } r = -\dfrac{b^{T}Ab}{M\|b\|^2} + \sqrt{\left( \dfrac{b^{T}Ab}{M\|b\|^2}\right)^2 + \dfrac{2\|b\|}{M}}.
	\vspace*{-0.2cm}
\end{equation*}
The results are presented in Table \ref{table:1}, showing the number of full iterations $k/N$ (ITER)  and CPU time in seconds (CPU).   We also report the number of function evaluations (FE) for the algorithms based on line-search. As one can see from Table \ref{table:1}, the randomized versions of our algorithms, RCGD  and RCPG, with $N=n$ are comparable and they are much faster than the cyclic counterparts or than the   algorithms in \cite{CarDuc:17}, \cite{Nes:19} and than those based on line search. Moreover, for $\Hf = 0.51 \cdot \|U^{T}_{k}AU_{k}\|$ the performance of RCGD is further improved.  From  Table \ref{table:1} one can also notice that coordinate descent methods have  better  performance on optimization problems having  the gap $\lambda_{1}(A) - \lambda_{2}(A)$  large. 

\vspace{-0.1cm}
	\scriptsize 
\begin{longtable}{| c| c | c| c| c | c | c | c |}
	\hline
	\multicolumn{8}{|c|}{ $B =$ diag($10^4$,randn(n-1,1)).} \\
	\hline
	\multicolumn{2}{|c|}{\text{n}} & \multicolumn{3}{|c|}{$10^3$} & \multicolumn{3}{|c|}{$10^4$}  \\ \hline
	\multicolumn{2}{|c|}{\text{M}} & 1 & 0.1 & 0.01 & 1 & 0.1 & 0.01  \\
	\hline
	\multirow{1}{*}{ \text{RCPG (N=n)}} & \text{ITER}  & 120  & 757  & 351 & 16 & 183  & 74 \\
	& \text{CPU} & 2.78  & 21.3  & 9.2  & 9.88 & 132.6 & 60.7 \\ \hline
	\multirow{1}{*}{\text{RCGD-1 (N=n)}} & \text{ITER} & 74 & 391  & 196 & 16 & 202  & 80 \\
	& \text{CPU} & \textbf{0.55}  & \textbf{3.42}  & \textbf{1.52} & \textbf{7.57} & \textbf{115.1}  & \textbf{43.6} \\ \hline
	\multirow{1}{*}{\text{RCGD-2 (N=n)}} & \text{ITER} & 130  & 668  & 306 & 17 & 202  & 103 \\
	& \text{CPU} & 0.98  & 5.87  & 2.43 & 7.58 & 115.3  & 67.7 \\ \hline	
	\multirow{1}{*}{ \text{CCPG (N=n)}} & \text{ITER}  & 120788 & 297971  & 377884 &  &   & \\
	& \text{CPU} & 1870.3  & 5678  & 6489.9 & ** & ** & ** \\ \hline
	\multirow{1}{*}{\text{CCGD-1 (N=n)}} & \text{ITER} & 866947  & 2188155  & 2642269 &  &   & \\
	& \text{CPU} & 1479.1  & 4410.6  & 4568.5 & ** & **  & ** \\ \hline
	\multirow{1}{*}{\text{CCGD-2 (N=n)}} & \text{ITER} & 120789  & 297973  & 377891 &  &   & \\
	& \text{CPU} & 205.4  & 595.3 & 593.6 & ** & **  & ** \\ \hline	
	\multirow{1}{*}{\text{CGD-1 (N=1)}} & \text{ITER} & 23055  & 236708  & 66166 & 7067 & 191359  & 94752 \\
	& \text{CPU} & 2.24 & 41.96  & 12.7 & 104.8 & 2743.3  & 1343.4 \\ \hline
	\multirow{1}{*}{\text{CGD-2 (N=1)}} &\text{ITER} & 45190  & 463992  & 129697 & 13850  & 375063 & 185717 \\
	& \text{CPU} & 4.3  & 81.5  & 30.05 & 203.9 & 5369.1  & 4581.3 \\\hline
	\multirow{1}{*}{\cite{CarDuc:17}} & \text{ITER} & 361383  & 3710762  & 1037241 & 110735 &   & \\
	& \text{CPU} & 54.6  & 665.4  & 217.03 & 1601.6 & **  & ** \\ \hline
	\multirow{1}{*}{\cite{Nes:19}} & \text{ITER} & 45190  & 463992  & 129697 & 13850 & 375063  & 185717 \\
	& \text{CPU} & 7.08  & 81.2  & 18.5 & 199 & 8100.5  & 4372.2 \\ \hline
		\multirow{1}{*}{\text{GD (line-search)}} & \text{ITER} & 10358  & 104810  & 29259 & 3083 &   & 42058 \\
	& \text{FE} & 135607  & 1371266  & 383168 & 40452 & **  & 550611 \\
	& \text{CPU} & 7.87  & 86.7  & 24.2 & 587.4 &   & 1349.8\\ \hline
	\multirow{1}{*}{\text{RCGD (line-search)}} & \text{ITER} & 151  & 918  & 437 & 23 & 299  & 161 \\
	& \text{FE} & 742  & 4259  & 2203 & 104 & 1201  & 674\\
	& \text{CPU} & 12.1  & 91.8 & 61.8 & 106.3 & 1405.6  & 719.4 \\ \hline
	\multirow{1}{*}{\text{CCGD (line-search)}} & \text{ITER} &   &   &  &  &   & \\
	& \text{FE} & **  & **  & **  & ** & ** & **\\
	& \text{CPU} &   &   &  &  &   & \\ \hline
	\hline
		\caption{Full iterations (ITER) and CPU time in seconds (CPU) for variants of algorithms CPG and CGD, algorithms in  \cite{CarDuc:17} and \cite{Nes:19} and line search based methods on cubic Newton subproblem.}
	\label{table:1}
\end{longtable}

\vspace*{-0.6cm}	

\normalsize

\noindent In the second set of experiment we want to find the smallest eigenvalue of an indefinite matrix $A$. As proved in \cite{CarDuc:17}, if a matrix $A$ has at least one negative eigenvalue, we can use the nonconvex formulation  \eqref{opt_prob} with $b = 0$ to find the smallest eigenvalue. We compare variants of our two algorithms (CPG and CGD) with  algorithm from \cite{Nes:19} and  the power method. We consider   $A$  the matrix c-30 of group Schenk\underline{ }IBMNA from University of Florida Sparse Matrix Collection \cite{DavHu}. The dimension of this matrix is $n=5321$. We denote $\lambda_k$ the eigenvalue along the iterates. In Figure \ref{fig:eig1} we plot the error $\|Ax_{k} - \lambda_kx_{k}\|$ and the value of $\lambda_k$ along time (in seconds) for our  algorithms RCPG,  RCGD-2, GD-2,  algorithm  \cite{Nes:19} and  power method. Clearly, the randomized coordinate descent variants ($N=n$) of our two algorithms CPG and CGD have superior performance compared to e.g., power method or the algorithm in \cite{Nes:19}.

\normalsize

	\vspace{-0.1cm}

\begin{figure}[!ht]
	\centering
	\includegraphics[width=6.5cm,height=5cm]{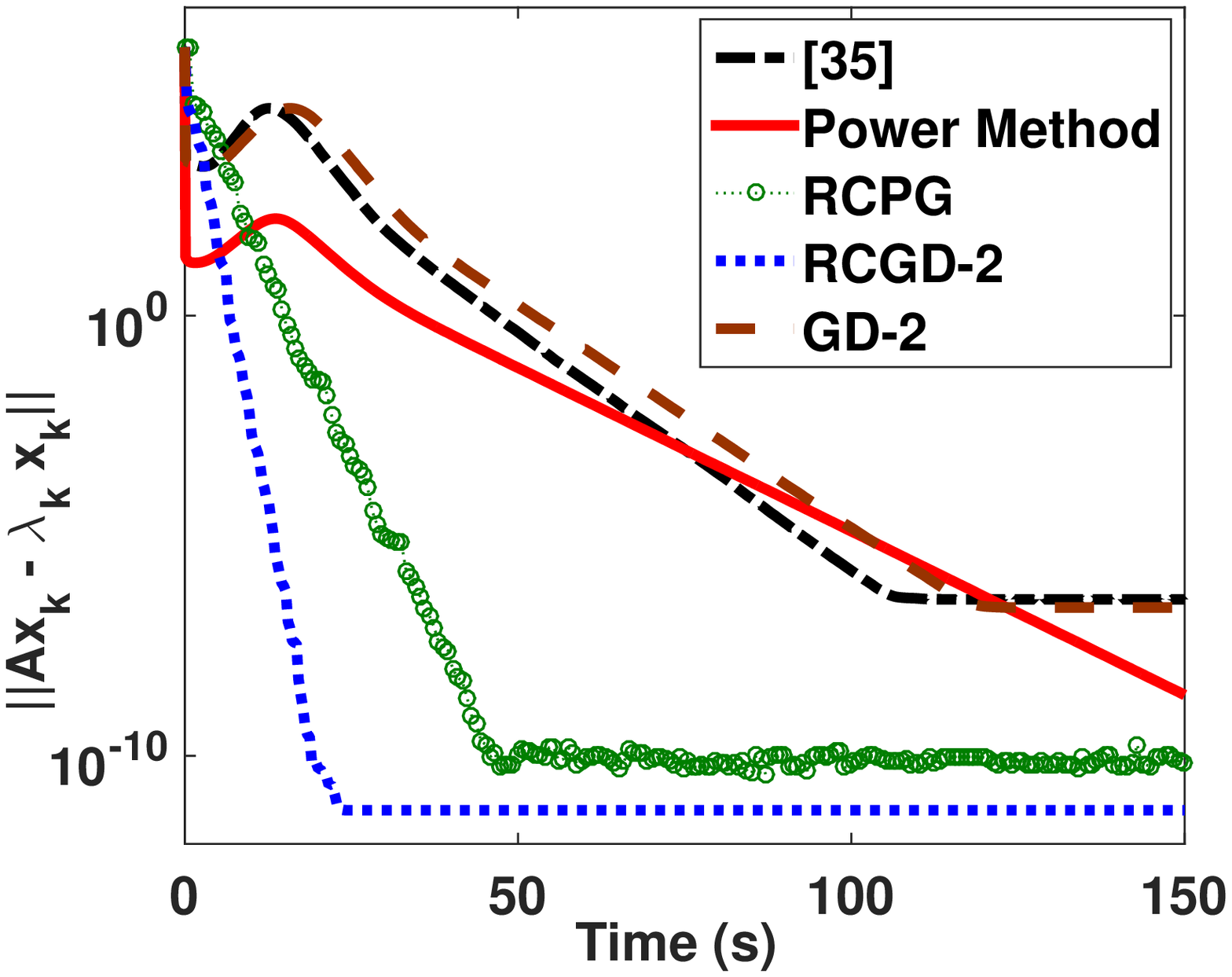}
	\hspace{-0.5cm}
	\includegraphics[width=6.5cm,height=5cm]{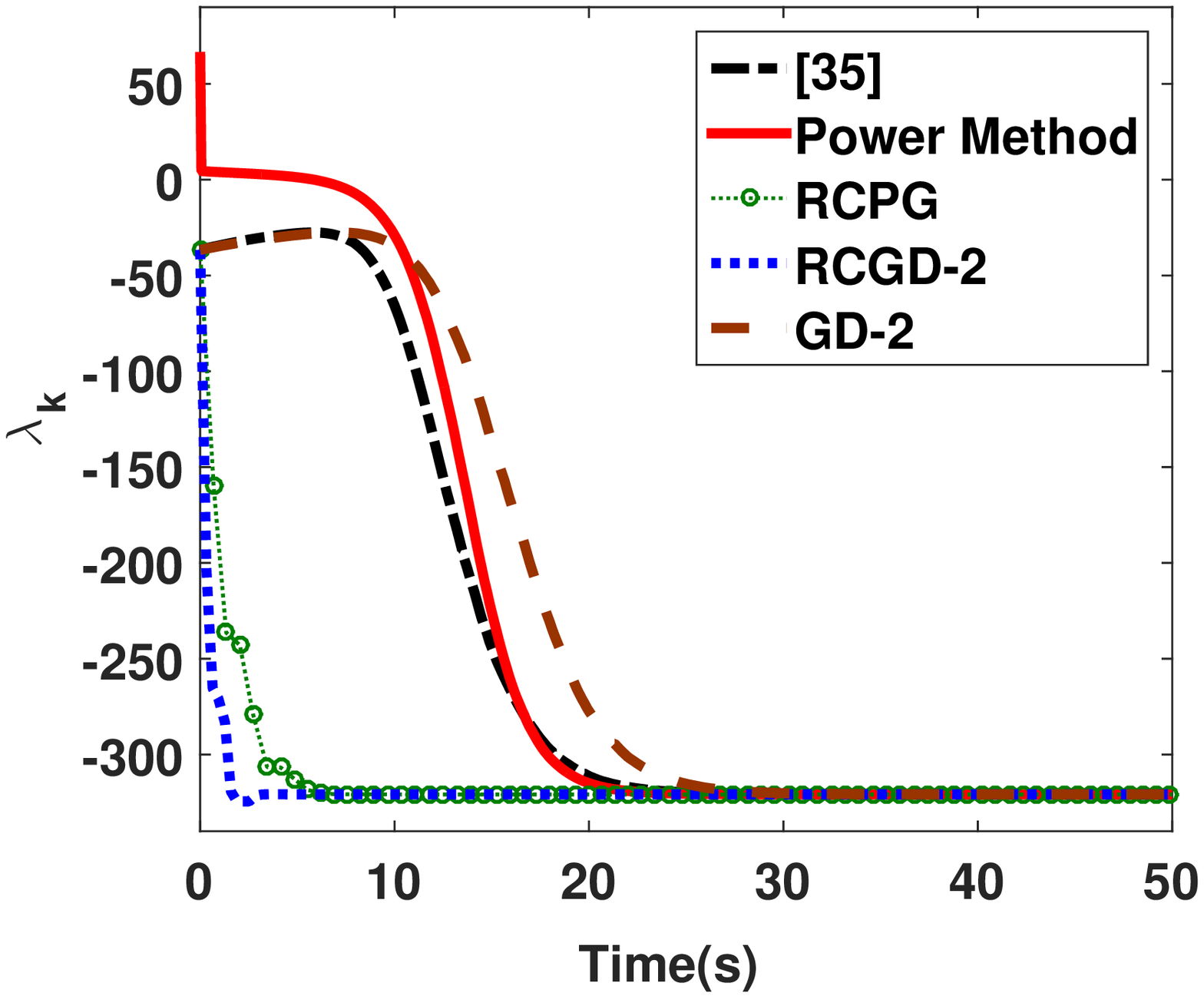} 
	\vspace*{-0.4cm}
	\caption{Behaviour along time (sec) of RCPG,  RCGD-2, GD-2,  algorithm  \cite{Nes:19} and  power method on group Schenk\underline{ }IBMNA,  matrix  c-30: left - $\|Ax_k- \lambda_k x_k\|$, right - $\lambda_k $.    }
	\label{fig:eig1}
\end{figure}

\subsection{Matrix factorization}
Finally, we consider the penalized orthogonal matrix factorization problem, see also \cite{AhoHie:21}:

\vspace{-0.3cm}

\begin{equation} 
\min_{(W,V)} F(W,V)   = \min_{(W,V)} \dfrac{1}{2} \| X - WV\|^{2}_{F} + \dfrac{ \lambda}{2} \|I-VV^{T}\|^2_{F},  \label{eq:MF}
\vspace{-0.2cm}
\end{equation}

\noindent with $W \in \mathbb{R}^{m\times r}$ and $V \in \mathbb{R}^{r \times n}$.  Let us define:
$
f(W,V) =  \dfrac{1}{2} \| X - WV\|^{2}_{F}$ and $ \psi(W,V) = \dfrac{ \lambda}{2} \|I-VV^{T}\|^2_{F}.$
\noindent Then, one can easily  compute:
\begin{equation*}
\nabla_{W} f(W,V) = WVV^{T} - XV^{T} \;\; \text{ and } \;\; \nabla^{2}_{WW} f(W,V)Z = ZVV^{T}.
\end{equation*}

\noindent Thus, $\nabla f$ is Lipschitz continuous w.r.t. $W$, with   $L_{1}(V) = \|VV^{T}\|_{F}$. Similarly: 
	\vspace*{-0.2cm}
\begin{equation*}
\nabla_{V} f(W,V) = W^{T}WV - W^{T}X \text{ and }  \nabla^{2}_{VV} f(W,V)Z = W^{T}WZ.
	\vspace*{-0.2cm}
\end{equation*}
\noindent Therefore, $\nabla f$ is also Lipschitz continuous w.r.t. $V$, with constant Lipschitz $L_{2}(W) = \|W^{T}W\|_{F}$. On the other hand, $\nabla_{V} \psi (W,V) = 2 \lambda (VV^{T}V - V)$ and thus
	\vspace*{-0.2cm}
\begin{equation*}
\nabla^2_{VV} \psi (W,V)Z = 2 \lambda (ZV^{T}V + VZ^{T}V + VV^{T}Z - Z). 
	\vspace*{-0.2cm}
\end{equation*}
\noindent Therefore, we get the following bound on the Hessian:
	\vspace*{-0.2cm}
\begin{equation*}
\langle Z, \nabla^2_{V} \psi (W,V)Z \rangle = \langle Z, 2 \lambda (ZV^{T}V + VZ^{T}V + VV^{T}Z - Z) \rangle \leq 6\lambda\|Z\|_{F}^{2} \|V\|_{F}^{2}. 
	\vspace*{-0.2cm}
\end{equation*}

\noindent This shows that $ \nabla^2_{V} \psi (\cdot)$ satisfies condition [A.4] in Assumption \ref{ass2}, with $p=2$ and $H_{\psi} = 6\lambda$. Note that if we assume that there exist $\bar{L}_{1}, \bar{L}_{2} > 0 $ such that $L_{1}(V) \leq \bar{L}_{1}$ and $L_{2}(W) \leq \bar{L}_{2}$, then   Lemmas \ref{lem:desRCGD2} and \ref{lem:RCGD2} are still valid. 
Therefore, we can solve  problem \eqref{eq:MF} using algorithm CGD with the second stepsize choice (i.e., equation \eqref{eq:HF2}) to update $V$. Moreover, since $\nabla F$ is Lipschitz continuous w.r.t. $W$, we can use algorithm CGD  to also update $W$. Since we have only 2 blocks we consider the cyclic variant of CGD, named CCGD. Thus, the iterations of algorithm  CCGD are: 
	\vspace*{-0.2cm}
\begin{eqnarray*}
	W_{k+1} &=& W_{k} - \dfrac{1}{\HfW(V_{k})}(W_{k}V_{k}V_{k}^{T} - XV_{k}^{T}), \\
	V_{k+1} &=& V_{k} - \dfrac{1}{\HF}\left( W_{k+1}^{T}W_{k+1}V_{k} - W_{k+1}^{T}X + 2 \lambda (V_{k}V^{T}_{k}V_{k} - V_{k})\right), 
\end{eqnarray*}

\noindent with $\HfW(V_{k}) > \dfrac{L_{1}(V_{k})}{2}$, $\HfV(W_{k+1}) > \dfrac{L_{2}(W_{k+1})}{2}$,   $\HF = 12 \lambda \|V_{k}\|_{F}^2 + 12\lambda \alpha_{k}^2 + \HfV(W_{k+1})$ and $\alpha_{k}$ is the positive root of the following third order equation:
	\vspace*{-0.2cm}
\begin{equation*}
12 \lambda \alpha^{3} + \left( 12\lambda  \|V_{k}\|_{F}^2  + \HfV(W_{k+1})\right) \alpha - \|\nabla_{V} f(W_{k+1},V_{k}) +\nabla_{V} \psi(W_{k+1},V_{k}) \|_{F} = 0.
	\vspace*{-0.2cm}
\end{equation*}

\noindent In our experiments, in CCGD-1 we take $\HfW(V_{k}) =  0.51 \cdot L_{1}(V_{k})$ and $\HfV(W_{k+1}) = 0.51 \cdot L_{2}(W_{k+1})$, while in CCGD-2 we take  $\HfW(V_{k}) = L_{1}(V_{k})$ and $\HfV(W_{k+1}) = L_{2}(W_{k+1})$. We compare the two variants of CCGD algorithm with the algorithm BMM in \cite{HiePha21}. For problem \eqref{eq:MF}, BMM is a Bregman type gradient descent method having computational cost per iteration comparable to our method. For numerical tests, we consider SalinasA and Indian Pines data sets from \cite{DataSet}. Each row of matrix $X$ is a vectorized image at a given band of the data set. Each image is normalized to  $[-1,1]$. The starting matrix  $W_{0}$ is generated from a standard normal distribution $\mathcal{N}(0,1)$ and the matrix $V_{0}$ is generated with orthogonal rows. Moreover, we take $\lambda = 1000$ and the dimension $r$ is taken as  in \cite{DataSet} (i.e.,  in  SalinasA  we take $r=6$ and in Indian Pines we choose $r = 16$). We run all the algorithms for $100$s. The results are displayed in Figures \ref{fig:1} (SalinasA) and \ref{fig:2} (Indian Pines), where we plot  the evolution of  function values (left) and the orthogonality error $O_{\text{error}} = \|I - V_{k}V_{k}^{T}\|_{F}$ (right) along time (in seconds). Note that in terms of function values  CCGD is competitive with algorithm BMM. However, our algorithm  identifies orthogonality faster than BMM. 

\begin{figure}[!ht]
\centering
	\includegraphics[width=1\textwidth,height=4.3cm]{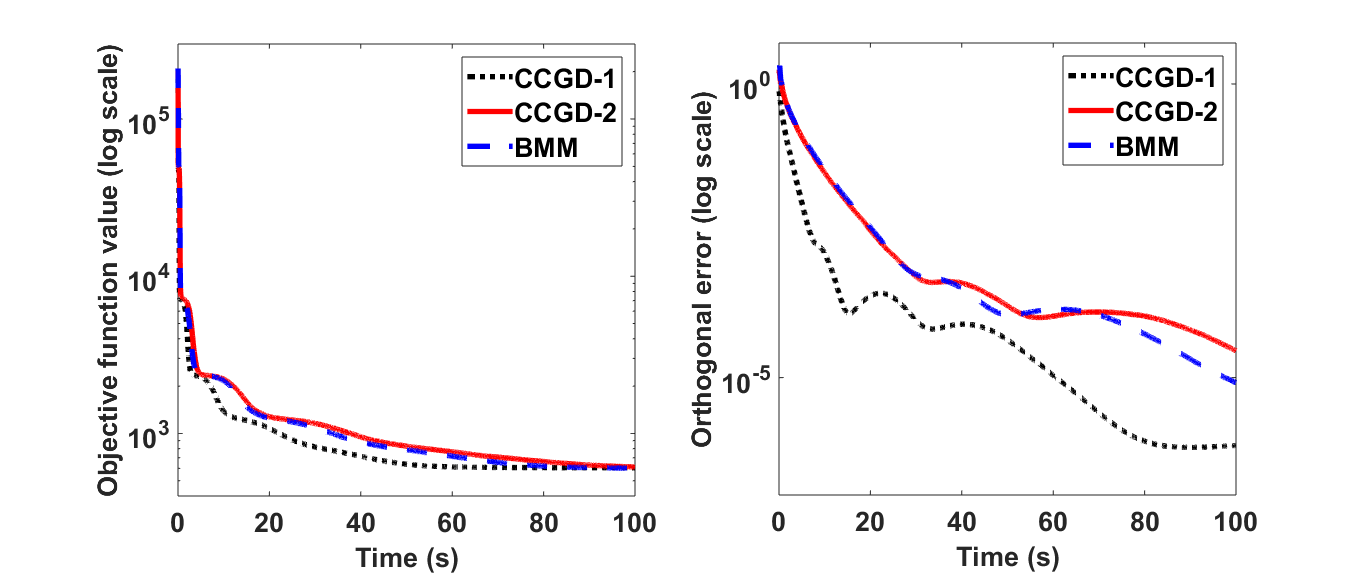} 
		\vspace*{-0.7cm}
\caption{CCGD and BMM on SalinasA: left - function values, right - orthogonality error }
\label{fig:1}
\end{figure}
\vspace*{-0.7cm}
\begin{figure}[!ht]
	\centering
	\includegraphics[width=1\textwidth,,height=4.3cm]{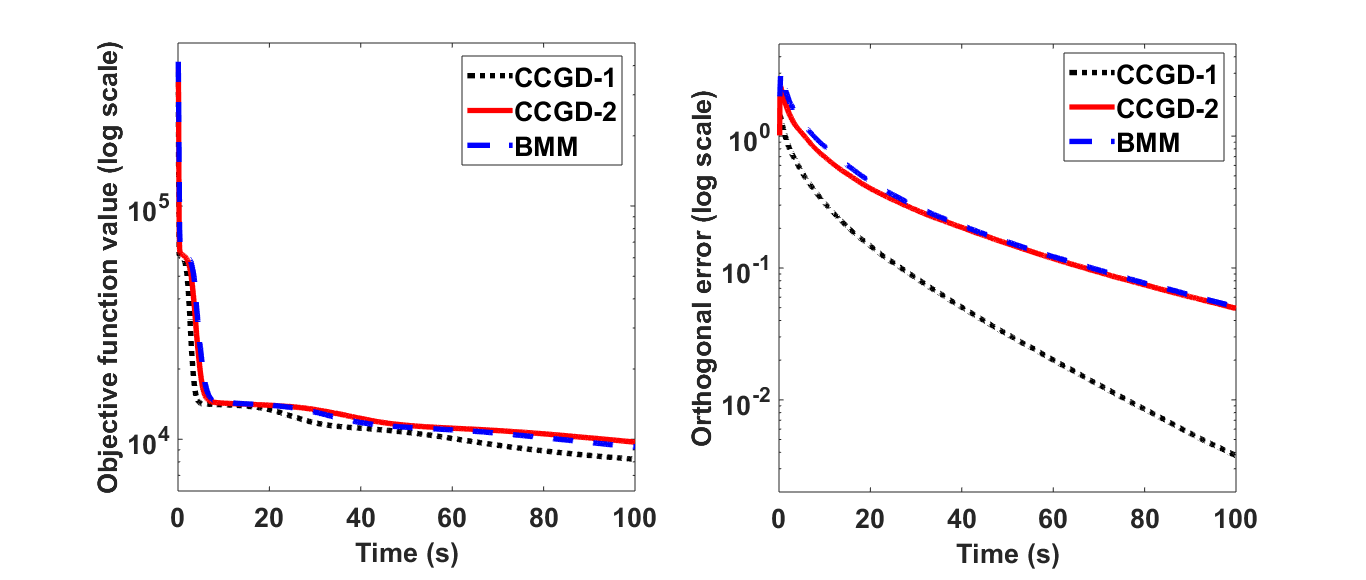} 
		\vspace*{-0.7cm}
	\caption{CCGD and BMM on Indian Pines: left - function values, right - orthogonality error.}
	\label{fig:2}
\end{figure}


\section{Conclusions}
In this paper we have considered composite problems having  the  objective  formed as a sum of two  terms, one smooth  and the other twice differentiable, both possibly nonconvex and nonseparable. For solving this  problem we have proposed two algorithms,  a coordinate proximal gradient   method and a coordinate gradient descent method,  respectively. For the second algorithm we have designed several novel  adaptive stepsize strategies which  guarantee descent. For both algorithms we derived  convergence bounds in both convex and nonconvex settings.   Preliminary numerical results confirm the  efficiency of our  algorithms on real applications.

\medskip
\noindent \textbf{Acknowledgements.}
The research leading to these results has received funding from: TraDE-OPT funded by the European Union’s Horizon 2020 Research and Innovation Programme under the Marie Skłodowska-Curie grant agreement No.  861137; 
UEFISCDI PN-III-P4-PCE-2021-0720, under project L2O-MOC, nr.  70/2022. %

\bibliographystyle{siamplain}
\bibliography{references}

\begin{thebibliography}{100}

\bibitem{AbeBec:21}
A. Aberdam and A. Beck \emph{An Accelerated Coordinate Gradient Descent Algorithm for Non-separable Composite Optimization}, J. Opt. Theory and Appl., 193:  219--246  , 2021. 

\bibitem{Arm:66}
L. Armijo. \emph{Minimization of functions having lipschitz continuous first partial derivatives.} Pacific J. Math., 16(1): 1–3, 1966.

\bibitem{AhoHie:21}
M. Ahookhosh, L.T. K. Hien, N. Gillis and P. Patrinos, \emph{ Multi-block Bregman proximal alternating linearized minimization and its application to orthogonal nonnegative matrix factorization}, Computational Optimization and Application, 79:  681--715, 2021.

\bibitem{Bec:14}
A.  Beck, \emph{On the convergence of alternating minimization for convex programming with applications to iteratively reweighted least squares and decomposition schemes}, SIAM Journal on Optimization, 25(1): 185--209, 2014.


\bibitem{BecTet:13}
A.  Beck  and  L. Tetruashvili,  \emph{On the convergence of block coordinate descent type methods}, SIAM Journal on Optimization, 23(4): 2037--2060, 2013.

\bibitem{Ber:99}
D.  Bertsekas, \emph{Nonlinear Programming}, Athena Scientific,  1999.

\bibitem{BolDan:07}
J. Bolte, A. Daniilidis and A. Lewis, \emph{The Łojasiewicz inequality for nonsmooth subanalytic functions with applications to subgradient dynamical systems},  SIAM Journal on Optimization, 17: 1205--1223, 2007.

\bibitem{BolSab:14}
J. Bolte,  S. Sabach and M. Teboulle,  \emph{Proximal alternating linearized minimization for nonconvex and nonsmooth problems}, Mathematical Programming,  146:  459--494, 2014. 

\bibitem{Bon:21}	
S. Bonettini, \emph{Inexact block coordinate descent methods with application to non-negative matrix factorization}, IMA Journal of Numerical Analysis, 31:  1431--1452, 2021.


\bibitem{CarDuc:17}
Y.  Carmon  and  J.C. Duchi, \emph{Gradient descent efficiently finds the cubic  regularized nonconvex Newton step}, SIAM Journal on Optimization, 29(3):  2146--2178, 2019. 


\bibitem{Dav:16}
D. Davis,  \emph{The asynchronous PALM algorithm for nonsmooth nonconvex problems},  arXiv preprint:  1604.00526, 2016.

\bibitem{DavHu}
T.A. Davis and Y. Hu. \emph{The University of Florida Sparse Matrix Collection}, ACM Transactions on Mathematical Software 38(1): 1--25, 2011.





\bibitem{FerRic:15}
O. Fercoq and P. Richtarik. \emph{Accelerated, parallel and proximal coordinate descent}, SIAM Journal on Optimization, 25(4): 1997--2023, 2015.


\bibitem{FriHasHofTib}
J. Friedman, T. Hastie, H. Hofling and R. Tibshirani, \emph{Pathwise coordinate optimization}, The Annals of Applied Statistics: 1(2): 302--332, 2007.


\bibitem{GriSci:00}
L. Grippo and  M. Sciandrone, \emph{On the convergence of the block nonlinear Gauss-Seidel method under convex constraints}, Operations Research Letters, 26(3): 127--136, 2000.

\bibitem{GriIut:21}
D. Grishchenko,  F. Iutzeler and  J. Malick \emph{Proximal gradient methods with adaptive subspace sampling}, Mathematics of Operations Research,  46(4): 1235--1657, 2021.


\bibitem{HanKon:18}
  F. Hanzely, K.  Mishchenko and P. Richtarik, \textit{SEGA: Variance Reduction via Gradient Sketching},  Advances in Neural Information Processing Systems, 31, 2018. 


\bibitem{HiePha21}
L.T.K. Hien, D.N.  Phan, N. Gillis, M. Ahookhosh and P. Patrinos \emph{Block  bregman majorization minimization with extrapolation},  SIAM Jounal on Optimization, 4(1):  1 -- 25, 2022

\bibitem{LatThe:21}
P. Latafat,  A. Themelis  and  P. Patrinos \emph{Block-coordinate and incremental aggregated proximal	gradient methods for nonsmooth nonconvex problems}, Math. Prog., 193: 195 -- 224, 2021. 


\bibitem{LuFreNes:18}
H. Lu,  R.M. Freund and Yu. Nesterov,  \emph{Relatively Smooth Convex Optimization by First-Order Methods and Applications},  SIAM Journal on Optimization, 28(1): 333--354, 2018.  

\bibitem{LuXia:14}
Z. Lu and L. Xiao,  \emph{On the complexity analysis of randomized block-coordinate
	descent methods}, Mathematical Programming, 152(1-2): 615--642, 2015.


\bibitem{MauFadAtt:22}
R. Maulen,  S.J. Fadili and H. Attouch,  \emph{An SDE perspective on stochastic convex optimization}, arXiv prepint: 2207.02750, 2022.   

\bibitem{Mes:82}
B. E. Meserve,  \emph{Fundamental Concepts of Algebra}, Dover, New York, 1982.


\bibitem{Mit:97}
T. Mitchell, \emph{Machine Learning},  McGraw Hill, 1997. 


\bibitem{Nec:13}
I. Necoara, \emph{Random coordinate descent algorithms for multi-agent convex optimization over networks}, IEEE Transactions on Automatic Control, 58(8):  2001--2012, 2013.

\bibitem{NecCho:21}
I. Necoara and  F. Chorobura, \emph{Efficiency of stochastic coordinate  proximal gradient  methods on nonseparable composite  optimization}, arXiv preprint: 2104.13370, 2021.

\bibitem{NecCli:13}
I. Necoara and  D. Clipici, \emph{Efficient parallel coordinate descent algorithm for convex optimization problems with separable constraints: Application to distributed MPC}, Journal of Process Control,  23(3):  243--253,  2013.

\bibitem{NecCli:16}
I. Necoara and D. Clipici, \emph{Parallel random coordinate descent methods for composite minimization: convergence analysis and error bounds}, SIAM J. Opt., 26(1): 197--226, 2016. 


\bibitem{NecTak:20}
I. Necoara and  M. Takac, \textit{Randomized sketch descent methods for non-separable linearly constrained optimization},  IMA Journal of Numerical Analysis, 41(2): 1056--1092, 2021. 

\bibitem{Nes:08}
Yu. Nesterov, \emph{Accelerating the cubic regularization of Newton’s
	method on convex problems}, Mathematical Programming, 112: 159--181, 2008.


\bibitem{Nes:10}
Yu.  Nesterov, \emph{Efficiency of coordinate descent methods on huge-scale optimization problems}, SIAM Journal on Optimization:  22(2): 341--362,  2012.




\bibitem{Nes:19}
Yu. Nesterov, \emph{Inexact basic tensor methods for some classes of convex optimization problems}, Optimization Methods and Software,  doi: 10.1080/10556788.2020.1854252, 2020.

\bibitem{NesPol:06} 
Yu.  Nesterov and  B.T.  Polyak, \emph{Cubic regularization of Newton method and its global performance}, Mathematical Programming, 108: 177--205, 2006. 

\bibitem{NesSti:17}
Yu. Nesterov and S.U. Stich, \emph{Efficiency of the accelerated coordinate descent method on structured optimization problems}, SIAM Journal on Optimization, 27(1): 110–123, 2017.

\bibitem{Pow:73}
M.J.D. Powell, \emph{On search directions for minimization algorithms}, Mathematical Programming, 4: 193–201, 1973.

\bibitem{RicTak:11}
P.  Richtarik  and  M. Takac, \emph{Iteration complexity of randomized block-coordinate descent methods for minimizing a composite function}, Mathematical Programming, 144: 1--38, 2014.

\bibitem{RobSie:71}
H. Robbins, and D. Siegmund, \emph{A convergence theorem for non-negative almost supermartingales and some applications},  Optimizing Methods in Statistics, 233–257, 1971.
\red{\bibitem{SteSha:05}
E. M. Stein and R. Shakarchi, \emph{Real Analysis: Measure Theory, Integration, and Hilbert Spaces}, Princenton University Press, 2005.}


\bibitem{TseYun:09}
P. Tseng and  S. Yun, \emph{Block-coordinate gradient descent method for linearly constrained nonsmooth separable optimization}, Journal of Optim.  Theory and  Applications,  140, 2009. 

\bibitem{Wri:15}
S.J. Wright, \emph{Coordinate descent algorithms}, Mathematical Programming,  151(1): 3--34, 2015.

\bibitem{DataSet}
 \url{http://www.ehu.eus/ccwintco/index.php/Hyperspectral_Remote_Sensing_Scenes. }
\end{thebibliography}


\section*{Appendix}	
\label{apendix}

\blue{ Recall that $\mathbbm{1}_{A}$  denotes the indicator function of the set $A$. 
\begin{lemma}
	\label{lem:prob}
	Let   $(X_{k})_{k \geq 0}$ be a sequence of random variables  on a  probability space ($\Omega$,$\mathcal{F}$,$P$).   Assume that exists  $C_{0} >  0$  such that $0 \leq X_{k}  \leq C_{0}$ with probability one for all $k \geq 0$. Let $\delta > 0$ and a measurable set   $\Omega_{\delta} \subset \Omega$ such that $P(\Omega_{\delta}) \geq 1- \delta$. Then:
	\vspace*{-0.1cm}
	\begin{equation}
	\mathbb{E}[X_{k}] - C_{0} \sqrt{\delta} \leq \mathbb{E}[X_{k} \mathbbm{1}_{\Omega_{ \delta}}] \leq \mathbb{E}[X_{k}] \quad  \forall k \geq 0.   \label{eq:136}
	\end{equation}	
\end{lemma}}
	\vspace*{-0.1cm} 
\begin{proof}	
\red{  Following an argument as in Lemma A.5 from \cite{MauFadAtt:22}, we have: 
	$	X_{k} = X_{k} \mathbbm{1}_{\Omega_{\delta}} +  X_{k}  \mathbbm{1}_{ \Omega \backslash  \Omega_{\delta}}.
$ This implies that: 
	\vspace*{-0.1cm} 
	\begin{equation}
		\mathbb{E}[  X_{k} ] = \mathbb{E}[  X_{k}  \mathbbm{1}_{\Omega_{\delta}}  ] + \mathbb{E}[X_{k}  \mathbbm{1}_{ \Omega \backslash  \Omega_{\delta}}]. \label{eq:134}
	\end{equation}	
	\noindent Using \blue{Cauchy- Schwarz} inequality,  $X_{k} \leq C_{0}$, and $\mathbb{E}[\mathbbm{1}_{ \Omega \backslash  \Omega_{\delta} }] = P(\Omega \backslash  \Omega_{\delta}) \leq \delta$, we get:
	\begin{equation}
		\mathbb{E}[X_{k} \mathbbm{1}_{ \Omega \backslash  \Omega_{\delta}}] \leq \sqrt{ \mathbb{E}[X_{k}^2]   } \sqrt{\mathbb{E}[\mathbbm{1}^2_{ \Omega \backslash  \Omega_{\delta}}]} \leq C_{0} \sqrt{\delta}. \label{eq:135}
	\end{equation}	
	\noindent From \eqref{eq:134} and \eqref{eq:135}, we get the left hand side in \eqref{eq:136}. Moreover, since $X_{k} \geq 0$, we have $\mathbb{E}[X_{k} \mathbbm{1}_{ \Omega \backslash  \Omega_{\delta}}] \geq 0$ and using \eqref{eq:134}, we get the right hand side in \eqref{eq:136}.} 
\end{proof}

\medskip

\textit{Proof of Lemma \ref{lemma:MVI}.} Consider a vector $u \in \mathbb{R}^{m}$, with $\|u\| =1$ and the parameterization $\alpha_{u}: [0,1] \to \mathbb{R}$ defined as	$\alpha_{u}(t) = \langle G\left(x+tUd\right) , u \rangle$.  From  mean value theorem, there exists $\bar{t} \in [0,1]$ such  that  	$\alpha_{u}(1) - \alpha_{u}(0) = \alpha_{u}'(\bar{t}) $.  This implies:
$
\langle G\left(x+Ud\right) - G\left(x\right)   , u \rangle  = \langle J\left( x+\bar{t}Ud\right)Ud, u \rangle. 	
$
 Hence 
 	\vspace*{-0.2cm}
\begin{equation*}
|\langle G\left(x+Ud\right) - G\left(x\right)   , u \rangle| \leq \| J\left( x+\bar{t}Ud\right)Ud\| \|u\| \leq  \| J\left( x+\bar{t}Ud\right)U\| \|d\| \|u\|.	
	\vspace*{-0.2cm}
\end{equation*}

\noindent If we define $y = x+\bar{t}Ud$ and take $u = \dfrac{G\left(x+Ud\right) - G\left(x\right)}{\|G\left(x+Ud\right) - G\left(x\right)\| } $, we get the statement.      \qed

\medskip
\textit{Proof of Lemma \ref{lem:recprob}.} We use  a similar definition as in \cite{RicTak:11}, i.e., let $\{\Delta_{k}^{\epsilon}\}_{k\geq 0}$ be the following sequence:
	\vspace*{-0.2cm}
	\begin{equation*}
	\Delta_{k}^{\epsilon} = \left\lbrace\begin{array}{ll} \Delta_{k}\; 
	\text{ if } \; \Delta_{k} \geq \epsilon, \quad \text{satisfies} \quad \Delta_{k}^{\epsilon} \leq \epsilon \iff \Delta_{k} \leq \epsilon \quad \forall k \geq \bar{k}. \\
	0 \; \text{ otherwise,}   
	\end{array}\right. 
\vspace*{-0.3cm}
	\end{equation*} 	
	
	\noindent Therefore, from Markov inequality, we have
	$\mathbb{P}\left[\Delta_{k} > \epsilon \right]  = \mathbb{P}\left[ \Delta_{k}^{\epsilon} > \epsilon \right]  \leq \dfrac{\mathbb{E}[\Delta_{k}^{\epsilon}]}{\epsilon}.$ Hence, it suffices to show that 
	$\theta_{K} \leq \epsilon\rho$, \noindent where $\theta_{k}:=\mathbb{E}[\Delta_{k}^{\epsilon}]$. If \eqref{eq:84} holds, then
		\vspace*{-0.2cm}
	\begin{equation*}
	\mathbb{E} [\Delta_{k+1}^{\epsilon} ] \leq \mathbb{E} [\Delta_{k}^{\epsilon}] - \mathbb{E} [ \Delta_{k}^{\epsilon}]^{\zeta + 1}, \quad
	\mathbb{E} [\Delta_{k+1}^{\epsilon}] \leq \left( 1 - \epsilon^{\zeta}\right) \mathbb{E} [\Delta_{k}^{\epsilon}]. 
		\vspace*{-0.2cm}
	\end{equation*}
Hence, we obtain $\theta_{k+1} \leq \theta_{k} - \theta_{k}^{\zeta+1} $ and $\theta_{k+1} \leq \left( 1 - \epsilon^{\zeta}\right)  \theta_{k} $. Using now the inequality (28) of Lemma 9 in \cite{NecCho:21}, we get 	$\left( k-\bar{k}\right) \zeta \leq \theta_{k}^{-\zeta} -  \theta_{\bar{k}}^{-\zeta}$. Therefore, if we let $k_{1} \geq \dfrac{1}{\zeta}\left(\dfrac{1}{\epsilon^{\zeta}}  - \dfrac{1}{ \Delta_{0}^{\zeta}}\right) + \bar{k}$, we obtain $\theta_{k_{1}} \leq \epsilon$. Finally, letting, $k_{2} \geq \dfrac{1}{\epsilon^{ \zeta}}\log\dfrac{1}{\rho}$, we have:
	\begin{equation*}
	\theta_{K} \leq \theta_{k_{1} + k_{2}} \leq \left( 1 - \epsilon^{\zeta}\right)^{k_{2}}  \theta_{k_1} \leq \left( (1-\epsilon^{\zeta})^{\frac{1}{\epsilon^{ \zeta}}}\right)^{\log\frac{1}{\rho}} \epsilon \leq \left( e^{-1} \right)^{\log\frac{1}{\rho}} \epsilon = \epsilon \rho, 
		\vspace*{-0.2cm}
	\end{equation*}
which proves our statement.

\end{document}